\documentclass[reqno]{amsart}      
\usepackage{amssymb,amsmath,graphicx}
\usepackage{enumerate}
\usepackage{xcolor}
\newtoks\prt 
  \numberwithin{equation}{section}
\newtheorem{thm}{Theorem}[section]

\newtheorem{ques}[thm]{Question} 
\newtheorem{lemma}[thm]{Lemma} 
\newtheorem{prop}[thm]{Proposition} 
\newtheorem{cor}[thm]{Corollary} 
\newtheorem{example}[thm]{Example}

\theoremstyle{definition} 
\newtheorem{example2}[thm]{Example}

\def\eqn#1$$#2$${\begin{equation}\label#1#2\end{equation}}

\def\B{\mathcal B} 
\def\C{\mathcal C}

\def\U{\mathcal U}

\def\M{\mathcal M}

\def\P{\mathcal{P}}

\def\co{\operatorname{conv}}

\def\en{\mathbb N} 
\def\er{\mathbb R} 
\def\qe{\mathbb Q} 
\def\ef{\mathbb F} 
 
\def\dist{\operatorname{dist}}

\def \Ienv{\operatorname{(I)-env}} 
\def \Icl{\operatorname{(I)-cl}} 
\def \Imn{\operatorname{(I)-mn}}
\def \Iw{\operatorname{(I)-ww}} 
\def \Ibl{\operatorname{(I)-bcl}} 
\def \Ibcl{\operatorname{(I)-cbcl}} 
\def \cl {\operatorname{cl}} 
 
\def \card {\operatorname{card}}

\def \Iccl {\operatorname{(I)-ccl}}

\def\spt{\operatorname{spt}} 
\def\lin{\operatorname{span}} 
 
\def\Ker{\operatorname{Ker}}

\def \reg {\partial _{\kern1pt\text{reg}}}

\def\ip#1#2{\left\langle#1,#2\right\rangle}

\def\di{\,\mbox{\rm d}}

\newcommand{\norm}[1]{\left\|#1\right\|}

\renewcommand{\Re}{\operatorname{Re}}

\newcommand{\wscl}[1]{\overline{#1}^{w^*}}

\newcommand{\abs}[1]{\left|#1\right|}
\newcommand{\setsep}{;\,}

\newcommand{\sesub}{\rotatebox{315}{$\hskip0pt\bigcup$}} 
\newcommand{\sesup}{\rotatebox{315}{$\bigcap$}}
\newcommand{\seeq}{\rotatebox{315}{$\hskip0pt\|$}} 
\begin{document}

\title{Topologies related to (I)-envelopes}

\author{Ond\v{r}ej F.K. Kalenda and Matias Raja}

\address{Charles University, Faculty of Mathematics and Physics, Department of
Mathematical Analysis, Sokolovsk{\'a} 86, 186 75 Praha 8, Czech Republic}
\email{kalenda@karlin.mff.cuni.cz}

\address{Departamento de Matem\'aticas, Universidad de Murcia, Campus de Espinardo, 30100 Espinardo, Murcia, Spain}
\email{matias@um.es}

\begin{abstract}
    We investigate the question whether the (I)-envelope of any subset of a dual to a Banach space $X$ may be described as the closed convex hull in a suitable topology. If $X$ contains no copy of $\ell^1$ then the weak topology generated by functionals of the first Baire class in the weak$^*$ topology works. On the other hand, if $X$ contains a complemented copy of $\ell^1$ or $X=\C([0,1])$ no locally convex topology works. If we do not require the topology to be locally convex, the problem is still open. We further introduce and compare several natural intermediate closure operators on a dual Banach space. Finally, we collect several intringuing open problems related to (I)-envelopes.
\end{abstract}



\thanks{The first author was partially supported by the Research grant GA \v{C}R 23-04776S. The second author was partially supported by the Research grants: Fundaci\'on S\'eneca - ACyT Regi\'on de Murcia project 21955/PI/22; and Grant PID2021-122126NB-C32  funded by MCIN/AEI/ 10.13039/501100011033 and by “ERDF A way of making Europe”, by the EU}

\maketitle

\section{Introduction}

If $X$ is a Banach space and $A$ is any subset of its dual $X^*$, the \emph{(I)-envelope} of $A$ is defined by
$$\Ienv(A)= \bigcap\left\{\overline{\co\bigcup_{n=1}^\infty \wscl{\co A_n}}^{\norm{\cdot}}\setsep A=\bigcup_{n=1}^\infty A_n
\right\}.$$
This notion was introduced in \cite{Ienv-israel}, inspired by the notion of (I)-generation from \cite{FL}. 
The quite strange definition is related to the representation of a convex set in terms of a 
\emph{James boundary}, such as the extreme points. Indeed, the Choquet-Bishop-de Leeuw theorem (cf. \cite[Theorem 3.81]{LMNS}) establishes that any point of a convex compact set of locally convex space can be represented as the barycenter of a Borel probability measure concentrated on the set of extreme points in the following subtle sense: every countable cover of the extreme point by closed sets has full measure one. That clarifies the use a countable decomposition in our main definition.
Moreover,  
(I)-envelopes were used to analyze James' characterization of reflexivity in \cite{Ienv-studia} or to investigate a quantitative version of the Grothendieck property in \cite{HB-groth,lechner-1gr}.

One of the main advantages of the notion of (I)-envelope is provided by the fact that it can be viewed as a geometric counterpart of Simons' equality (see \cite[Lemma 2.1 and Remark 2.2]{Ienv-israel} or Lemma~\ref{L:separ} below).

The prefix (I) means \emph{intermediate} as it is easy to see that
\begin{equation}\label{eq:norm-Ienv-w*}
\overline 
{\co A}^{\norm{\cdot}}\subset\Ienv(A)\subset\wscl{
\co A}\end{equation}
for any $A\subset X^*$. In this paper we address the following natural question.

\begin{ques}\label{q1}
Let $X$ be a Banach space. Is there a (locally convex) topology $\tau$ on $X^*$ such that $\Ienv(A)=\overline{\co A}^\tau$ for each $A\subset X^*$?
\end{ques}

In some case the answer is easily seen to be positive. For example, if $X^*$ is separable, then $\tau$ can be the norm topology (cf. \cite[Remark 1.1(ii)]{Ienv-israel}). In this paper we give some results in the positive direction and some results in the negative direction.

We point out that the above question has two versions -- the locally convex one and the general one. It turns out that these two versions are quite different and require different methods of proofs.

In Section~\ref{sec:LCS} we address the locally convex setting. In this case the Hahn-Banach separation theorem provides a powerful tool to determine the topology. Indeed, using among others Lemma~\ref{L:separ} and the Hahn-Banach theorem we arrive to Proposition~\ref{P:uniqueness} which shows that there is essentially a unique candidate for such a topology. Moreover, this topology is described without any reference to (I)-envelopes, so the answer for a concrete space can be obtained by checking whether this canonical topology works or not. Therefore we are able to show that the answer is negative in general (see Example~\ref{ex:ell1}), but it is positive in some special cases (see Theorem~\ref{T:X not containing ell1}). 

Section~\ref{sec:topologies} is then devoted to the general case, without the requirement of local convexity. In this case the situation is different because we cannot use the Hahn-Banach theorem. It seems that the natural approach is to use suitable closure operators. In Section~\ref{sec:topologies} we define several natural closure operators
inspired by (I)-envelopes. It turns out that some canonical ones do not work (see Example~\ref{ex:Pd[0,1]}). Again,  there is an essentially unique candidate (see Theorem~\ref{T:characterization}). However, this closure operator is defined using (I)-envelopes  and hence it is not easy to work with. Therefore the general question remains open. 

While investigating the general case of Question~\ref{q1} we discovered there is a rich variety of `intermediate topologies' defined by natural closure operators. A large part of Section~\ref{sec:topologies} is devoted to analyzing and comparing them.

In the last section we collect several intriguing problems concerning (I)-envelopes, not necessarily related to Question~\ref{q1}.

\section{Preliminaries}

We start by recalling a key lemma characterizing the (I)-envelope via a separation by a bounded sequence. Note that the equivalence $(1)\Leftrightarrow(4)$ provides the above-mentioned relationship to the Simons equality.

\begin{lemma}\label{L:separ}
Let $X$ be a Banach space, $A\subset X^*$ and $\eta\in X^*$. The following assertions are equivalent.
\begin{enumerate}[$(1)$]
    \item $\eta\notin \Ienv(A)$.
    \item There is a sequence $(x_n)$ in $B_X$ such that
    $$\sup_{\xi\in A}\limsup_{n\to\infty}\Re\xi(x_n)<\inf_{n\in\en}\Re\eta(x_n).$$
    \item There is a sequence $(x_n)$ in $B_X$ such that
    $$\sup_{\xi\in A}\limsup_{n\to\infty}\Re\xi(x_n)<\liminf_{n\to\infty}\Re\eta(x_n).$$
    \item There is a sequence $(x_n)$ in $B_X$ such that
    $$\sup_{\xi\in A}\limsup_{n\to\infty}\Re\xi(x_n)<\limsup_{n\to\infty}\Re\eta(x_n).$$
\end{enumerate}
\end{lemma}

This lemma is proved in \cite[Lemma 2.1]{Ienv-israel} for real Banach spaces. The complex setting follows by looking at the complex Banach space as to a real one. We are going to give a proof of (1)$\Rightarrow$(4) using only the Hahn-Banach theorem (the proof given in \cite{Ienv-israel} uses a minimax theorem from \cite{simons-minimax}).

\begin{proof}[Proof of implication $(1)\Rightarrow(4)$.]
Assume $\eta\notin\Ienv(A)$. Then there is a sequence $(A_n)$ of sets such that $A_n\nearrow A$ and
$$\eta\notin\overline{\bigcup_{n=1}^\infty \wscl{\co A_n}}^{\norm{\cdot}}.$$
Hence, 
$d=\dist(\eta,\bigcup_{n=1}^\infty \wscl{\co A_n})>0$. For each $n\in\en$ we have $(\eta+\frac d2B_{X^*})\cap \wscl{\co A_n}=\emptyset$.
Since these two sets are convex and weak$^*$-closed, and, moreover, the first one is weak$^*$-compact, by the Hahn-Banach theorem we get $x_n\in B_X$ with
$$\inf \left\{\Re\xi(x_n)\setsep \xi\in \eta+\frac d2 B_{X^*}\right\}>\sup \{\Re\xi(x_n)\setsep \xi\in \wscl{\co A_n}\},$$
i.e.,
$$\Re\eta(x_n)-\frac d2>\sup_{\xi\in A_n}\Re\xi(x_n).$$

Let $\xi\in A$ be arbitrary. Then there is $n_0\in\en$ such that $\xi\in A_n$ for $n\ge n_0$. Then
$$\Re\xi(x_n)<\Re\eta(x_n)-\frac d2$$
for $n\ge n_0$. Thus
$$\limsup_{n\to\infty}\Re\xi(x_n)\le \limsup_{n\to\infty}\Re\eta(x_n)-\frac d2.$$
It follows that the inequality from (4) is fulfilled.
\end{proof}

As an easy consequence we obtain an improvement of \eqref{eq:norm-Ienv-w*}. To formulate it we need the following notation.
If $X$ is a Banach space, we set
\begin{equation}
    \begin{aligned}
    B_1(X)&=\{x^{**}\in X^{**}\setsep \exists (x_n)\mbox{ a sequence in }X: x_n\stackrel{w^*}{\longrightarrow}x^{**}\},\\
 C(X)&=\{x^{**}\in X^{**}\setsep\exists C\subset X\mbox{ countable}: x^{**}\in\wscl{C}\}.
    \end{aligned}
\end{equation}

Recall that, given a Banach space $X$ and a subspace $Y\subset X^*$, the symbol $\sigma(X,Y)$ denotes the \emph{weak topology on $X$ induced by $Y$}, i.e., the weakest topology on $X$ making all elements of $Y$ continuous. This topology is clearly locally convex and $Y$ is exactly the space of all $\sigma(X,Y)$-continuous functionals (cf. \cite[Proposition 3.22]{Fab-BST}).

\begin{prop}\label{P:Ienv between closures}
Let $X$ be a Banach space and $A\subset X^*$. Then
$$\overline{\co A}^{\sigma(X^*,C(X))}\subset\Ienv(A)\subset\overline{\co A}^{\sigma(X^*,B_1(X))}.$$
\end{prop}

\begin{proof}
We will use Lemma~\ref{L:separ} and  the Hahn-Banach theorem.

Assume that $\eta\notin \Ienv(A)$. By Lemma~\ref{L:separ} there is a sequence $(x_n)$ in $B_X$ with
$$\sup_{\xi\in A}\limsup_{n\to\infty}\Re\xi(x_n)<\inf_{n\in\en}\Re\eta(x_n).$$
Let $x^{**}$ be any weak$^*$-cluster point of $(x_n)$ in $B_{X^{**}}$ (it exists by the Banach-Alaoglu theorem, cf. \cite[Theorem 3.37]{Fab-BST}). Then $x^{**}\in C(X)$. Moreover,
$$\Re x^{**}(\eta)\ge\inf_{n\in\en}\Re\eta(x_n)>\sup_{\xi\in A}\limsup_{n\to\infty}\Re\xi(x_n)\ge\sup_{\xi\in A}\Re x^{**}(\xi),$$
hence  $\eta\notin\overline{\co A}^{\sigma(X^*,C(X))}$. This completes the proof of the first inclusion.

Next suppose that $\eta\notin\overline{\co A}^{\sigma(X^*,B_1(X))}$. By the Hahn-Banach separation theorem there is $x^{**}\in B_1(X)$ such that
$$\Re x^{**}(\eta)>\sup_{\xi\in A}\Re x^{**}(\xi).$$
Since $x^{**}\in B_1(X)$, there is a sequence $(x_n)$ in $X$ with 
$x_n\stackrel{w^*}{\longrightarrow}x^{**}$. By the uniform boundedness principle this sequence is bounded (cf. \cite[Corollary 3.86]{Fab-BST}), so, up to multiplying it by a positive constant we may assume that $x_n\in B_X$ for $n\in\en$.
Then
$$\sup_{\xi\in A}\limsup_{n\to\infty}\Re\xi(x_n)=\sup_{\xi\in A}\Re x^{**}(\xi)<\Re x^{**}(\eta)=\limsup_{n\to\infty}\Re\eta(x_n),$$
hence $\eta\notin\Ienv(A)$ by Lemma~\ref{L:separ}. This completes the proof of the second inclusion.
\end{proof}

The next lemma provides a consequence of Lemma~\ref{L:separ} for linear subspaces of $X^*$.

\begin{lemma}\label{L:Ienv for linear subspace}
Let  $X$ be a Banach space.
\begin{enumerate}[$(a)$]
    \item Let $Y\subset X^*$ be a linear subspace. Then
    $$\Ienv(Y)=\{\eta\in X^*\setsep \forall (x_n)\mbox{ sequence in }B_X: x_n\stackrel{\sigma(X,Y)}{\longrightarrow}0\Rightarrow \eta(x_n)\to 0\}.$$
    \item Assume that $x^{**}\in X^{**}$. Then
    $$\Ienv(\ker x^{**})=\begin{cases}\ker x^{**}& x^{**}\in B_1(X)\\ X^*& x^{**}\notin B_1(X).\end{cases}$$
\end{enumerate}
\end{lemma}

\begin{proof}
$(a)$ This is a mild generalization of \cite[Lemma 1.2]{Ienv-studia}. We observe that
$$\sup_{\xi\in Y}\limsup_{n\to\infty}\Re\xi(x_n)=\begin{cases}0 & \mbox{ if }x_n\stackrel{\sigma(X,Y)}{\longrightarrow}0,\\+\infty &\mbox{ otherwise}
\end{cases}$$
and then use Lemma~\ref{L:separ}.

$(b)$ Assume first that $x^{**}\in B_1(X)$. Then $\ker x^{**}$ is $\sigma(X^*,B_1(X))$-closed, thus
$$\ker x^{**}\subset \Ienv(\ker x^{**})\subset \overline{\ker x^{**}}^{\sigma(X^*,B_1(X))}=\ker x^{**},$$
where the second inclusion follows from Proposition~\ref{P:Ienv between closures}.

Conversely, assume that $\Ienv(\ker x^{**})=\ker x^{**}$. If $x^{**}=0$, then $x^{**}\in B_1(X)$. So, assume $x^{**}\ne0$. Take $\eta\in X^*$ with $x^{**}(\eta)=1$. By $(a)$ there is a sequence $(x_n)$ is $B_X$ such that
$x_n\stackrel{\sigma(X,\ker x^{**})}{\longrightarrow}0$ but $\eta(x_n)\not\to0$. Hence, up to passing to a subsequence we may assume
that $\eta(x_n)\to c\ne0$. Then $\frac1c x_n\stackrel{w^*}{\longrightarrow}x^{**}$, hence $x^{**}\in B_1(X)$.
This completes the proof.
\end{proof}

\section{The case of locally convex topologies}\label{sec:LCS}

In this section we address the version of Question~\ref{q1} dealing with locally convex topologies. The first observation is the following consequence of Lemma~\ref{L:Ienv for linear subspace}.

\begin{cor}\label{cor:uniqueness}
Let $X$ be a Banach space.  Assume that there is a locally convex topology $\tau$ on $X^*$ such that $\Ienv(A)=\overline{\co A}^\tau$ for each $A\subset X^*$. Then $(X^*,\tau)^*=B_1(X)$. 
\end{cor}

\begin{proof}
If $x^{**}\in B_1(X)$, by Lemma~\ref{L:Ienv for linear subspace}$(b)$ we have $\Ienv(\ker x^{**})=\ker x^{**}$, 
hence $\ker x^{**}$ is $\tau$-closed, i.e., $x^{**}$ is $\tau$-continuous (cf. \cite[\S15.9.(1)]{kothe}).

Conversely, assume $x^{**}\in (X^*,\tau)^*$. Then $\ker x^{**}$ is $\tau$-closed, hence $\Ienv(\ker x^{**})=\ker x^{**}$. By \eqref{eq:norm-Ienv-w*} we deduce that $\ker x^{**}$ is norm-closed, hence $x^{**}\in X^{**}$. By Lemma~\ref{L:Ienv for linear subspace}$(b)$ we conclude that $x^{**}\in B_1(X)$.
\end{proof}

Using the previous corollary and Mazur's theorem (cf. \cite[Theorem 3.45]{Fab-BST}) we get the following equivalence.

\begin{prop}\label{P:uniqueness}
Let $X$ be a Banach space. The following assertions are equivalent.
\begin{enumerate}[$(1)$]
    \item There is a locally convex topology $\tau$ on $X^*$ such that $\Ienv(A)=\overline{\co A}^\tau$ for each $A\subset X^*$.
    \item $\Ienv(A)=\overline{\co A}^{\sigma(X^*,B_1(X))}$ for each $A\subset X^*$.
\end{enumerate}
\end{prop}

We continue by a partial positive answer to Question~\ref{q1}.

\begin{thm}\label{T:X not containing ell1}
Assume that $X$ is a Banach space not containing an isomorphic copy of $\ell^1$. Then $B_1(X)=C(X)$, hence for any set $A\subset X^*$ we have
$$\Ienv{A}=\overline{\co A}^{\sigma(X^*,B_1(X))}.$$
\end{thm}

\begin{proof}
The formula for the (I)-envelope follows from the equality $B_1(X)=C(X)$ and Proposition~\ref{P:Ienv between closures}. So, it is enough to prove the equality. Since obviously $B_1(X)\subset C(X)$, it suffices to prove the converse inclusion. To this end fix $x^{**}\in C(X)$. It follows that there is a countable set $C\subset X$ with $x^{**}\in\wscl{C}$. Set
$Y=\overline{\lin C}$. Then $Y$ is a separable subspace of $X$, hence it is a separable Banach space not containing an isomorphic copy of $\ell^1$. By the main result of \cite{odell-rosenthal} we have $B_1(Y)=Y^{**}$. Since $Y^{**}$ is canonically identified with $\wscl Y$ and $x^{**}\in\wscl Y$, we deduce that $x^{**}\in B_1(Y)\subset B_1(X)$.
\end{proof}

The next result is the basic counterexample to the `locally convex' variant of Question~\ref{q1}.

\begin{example}\label{ex:ell1}
There is no locally convex topology $\tau$ on $(\ell^1)^*$ such that $\Ienv(A)=\overline{\co A}^\tau$ for each $A\subset (\ell^1)^*$.
\end{example}

\begin{proof}
Recall that the space $\ell^1$ is weakly sequentially complete (cf. \cite[Theorem 5.36]{Fab-BST}), i.e., $B_1(\ell^1)=\ell^1$. Hence the topology $\sigma((\ell^1)^*,B_1(\ell^1))$ coincides with the weak$^*$-topology on $(\ell^1)^*$.

Recall further that $(\ell^1)^*=\ell^\infty$ (cf. \cite[Proposition 2.16]{Fab-BST}). Let $A=c_0$ be the canonical copy of $c_0$ in $\ell^\infty$. Then $\wscl{A}=\ell^\infty$ (by the Goldstine theorem \cite[Theorem 3.96]{Fab-BST}) but $\Ienv{A}=A=c_0$ as $A$ is separable. Hence we conclude using Corollary~\ref{cor:uniqueness}.
\end{proof}

One of the important tools in the previous example was the weak sequential completeness of $\ell^1$. It turns out that this may be used to provide a counterexample in a more general situation.

\begin{prop}
Let $X$ be a weakly sequentially complete Banach space.
The following are equivalent.
\begin{enumerate}[$(1)$]
    \item There is a locally convex topology $\tau$ on $X^*$ such that $\Ienv(A)=\overline{\co A}^\tau$ for each $A\subset X^*$.
    \item Any closed norm-separable subspace of $X^*$ is weak$^*$-closed.
    \item $X$ is reflexive.
\end{enumerate}
\end{prop}

\begin{proof}
$(3)\Rightarrow(1)$ This follows for example from \eqref{eq:norm-Ienv-w*} ($\tau$ may be the norm topology).

$(1)\Rightarrow(2)$ Since $X$ is weakly sequentially complete, we get $B_1(X)=X$, hence by Proposition~\ref{P:uniqueness} we may assume $\tau=w^*$.
Since $\Ienv{Y}=Y$ for any closed separable subspace of $X^*$, we deduce that any closed separable subspace of $X^*$ is also weak$^*$-closed. 

$(2)\Rightarrow(3)$ To prove that $X$ is reflexive it is enough to show that $X^*$ is reflexive (cf. \cite[Proposition 3.112]{Fab-BST}).
It follows from Eberlein-\v{S}mulyan theorem that it suffices to prove that any separable subspace of $X^*$ is reflexive  (cf. \cite[Theorems 3.109 and 3.111]{Fab-BST}).

So fix a closed separable subspace $Y\subset X^*$. By the assumption we know that $Y$ is weak$^*$-closed, so $Y=Z^*$ where $Z=X/Y_\perp$ (this follows by combining \cite[Proposition 2.6]{Fab-BST} with the bipolar theorem \cite[Theorem 3.38]{Fab-BST}). If $Y$ is not reflexive, there is some $z^{**}\in Z^{**}\setminus Z$. Then $\ker z^{**}$ is a weak$^*$-dense subpace of $Z^*=Y$. On the other hand, $\ker z^{**}$ is a closed separable subspace of $X^*$, so it is weak$^*$-closed by the assumption.
This is a contradiction. (Note that on $Y$ the topologies $\sigma(Z^*,Z)$ and $\sigma(X^*,X)$ coincide.)
\end{proof}

Another generalization of Example~\ref{ex:ell1} can be proved using the following easy statement.

\begin{lemma}\label{L:complemented}
Let $X$ be a Banach space. Let $Y\subset X$ be a complemented subspace and let $P:X\to Y$ be a projection witnessing it. Let $P^*:Y^*\to X^*$ be the adjoint map. Then the following assertions hold.
\begin{enumerate}[$(a)$]
    \item $P^*$is an isomporhic embedding, which is, moreover,  weak$^*$-to-weak$^*$   and $\sigma(Y^*,B_1(Y))$-to-$\sigma(X^*,B_1(X))$ homeomorphic.
    \item $\Ienv{P^*(A)}=P^*(\Ienv{A})$ for each $A\subset Y^*$.
\end{enumerate}
\end{lemma}

\begin{proof}
We start by observing that the space $X$ is canonically isomorphic to $Y\oplus Z$, where $Z=\ker P$. In this representation we have $P(y,z)=y$ for $(y,z)\in Y\oplus Z$. Then $X^*$ is isomorphic to $Y^*\oplus Z^*$, where $P^*:Y^*\to X^*$ is defined by $\eta\mapsto(\eta,0)$. Further, $X^{**}$ is canonically isomorphic to $Y^{**}\oplus Z^{**}$.

Now we are ready to prove all the statements.

$(a)$ By the above representation $P^*$ is an isomorphic embedding. It is a weak$^*$-to-weak$^*$ homeomorphism since  $(X^*,w^*)$ is obviously homeomorphic to $(Y^*,w^*)\times (Z^*,w^*)$.

Finally, since clearly $B_1(Y\oplus Z)\cap Y^{**}\times\{0\}=B_1(Y)\times\{0\}$, we deduce that $P^*$ is also a $\sigma(Y^*,B_1(Y))$-to-$\sigma(X^*,B_1(X))$ homeomorphism.

$(b)$ This follows easily from the definition of the (I)-envelope using the fact that $P^*$ is a linear injection which is a homeomorphism both in the norm and in the weak$^*$-topologies and, morerover, its range is weak$^*$-closed.
\end{proof}

Combining Example~\ref{ex:ell1} with Lemma~\ref{L:complemented} we immediately deduce the following result.

\begin{prop}\label{P:complemented ell1}
Let $X$ be a Banach space which contains a complemented isomorphic copy of $\ell^1$. Then there is a closed separable subspace $Y\subset X^*$ such that
$$Y\subsetneqq\overline{Y}^{\sigma(X^*,B_1(X))}.$$
In particular, there is no locally convex topology $\tau$ on $X^*$ such that $\Ienv(A)=\overline{\co A}^\tau$ for each $A\subset X^*$.
\end{prop}

If we compare Theorem~\ref{T:X not containing ell1} with Proposition~\ref{P:complemented ell1}, we see that they do not cover all Banach spaces. It is not clear what happens if a Banach space contains an isomorphic copy of $\ell^1$, but not a complemented one. We will settle the problem for spaces $\C(K)$ of continuous functions on a metrizable compact space $K$.

\begin{thm}\label{T:CK}
Let $K$ be an uncountable metrizable compact space. Then there is no locally convex topology $\tau$ on  $\C(K)^*$  such that $\Ienv(A)=\overline{\co A}^\tau$ for each $A\subset \C(K)^*$.
\end{thm}

\begin{proof}
Recall that, due to the Riesz theorem, $\C(K)^*$ is canonically isometric to $\M(K)$, the space of (signed or complex) Radon measures on $K$ equipped with the total variation norm. The bidual $\C(K)^{**}=\M(K)$ is not easy to describe in general, but it contains a nice subspace which may be easily described. It is the space $\B_b(K)$ of all bounded Borel-measurable functions on $K$ equipped by the supremum norm. If $g\in \B_b(K)$, it acts on $\C(K)^*=\M(K)$ by
$$\mu\mapsto \int g\di\mu.$$
Then $\C(K)$ is a closed subspace of $\B_b(K)$ and this inclusion corresponds to the canonical embedding of $\C(K)$ into $\C(K)^{**}$. Moreover, $B_1(\C(K))$ coincides with $\B_1^b(K)$, the space of bounded Baire-one functions on $K$ considered as a subspace of $\B_b(K)$.

Now we proceed with the proof itself. By the Miljutin theorem \cite{miljutin} $\C(K)$ is isomorphic to $\C([0,1])$, so we may assume without loss of generality that $K=[0,1]$. Denote by $\lambda$ the Lebesgue measure on $[0,1]$. Then we have:
\begin{equation}\label{eq:Fsigma}\begin{gathered}
 \mbox{There is an $F_\sigma$ set $E\subset [0,1]$ such that
 for any nonempty open set $G\subset [0,1]$}\\ \mbox{we have both $\lambda(E\cap G)>0$ and $\lambda(G\setminus E)>0$.}
\end{gathered}\end{equation}
Indeed, set $H=[0,1]\setminus\qe$. Then $H$ is a dense $G_\delta$-subset of $[0,1]$ with empty interior and full measure. Let $(G_n)$ be a countable open base of $[0,1]$ consisting of nonempty sets. Using regularity of the Lebesgue measure (and the fact that  $H$ is locally non-compact) we may construct by induction compact sets $F_{n,1}, F_{n,2}$ for $n\in\en$ such that for each $n\in\en$ we have
\begin{itemize}
    \item $\lambda(F_{n,1})>0$ and $\lambda(F_{n,2})>0$;
    \item $F_{n,1}\subset (G_n\cap H)\setminus \bigcup_{k<n}(F_{k,1}\cup F_{k,2})$;
    \item $F_{n,2}\subset (G_n\cap H)\setminus\left(F_{n,1}\cup \bigcup_{k<n}(F_{k,1}\cup F_{k,2})\right)$.
\end{itemize}
It is enough to take $E=\bigcup_{n\in\en}F_{n,1}$.

Next we set
$$Y=\{\mu\in \M([0,1])\setsep \mu \ll \lambda \ \&\ \mu(E)=0\},$$
where $\ll$ denotes absolute continuity.
Then $Y$ is a norm-closed linear subspace of $\M([0,1])$. Moreover, since the 
space $\{\mu\in \M([0,1])\setsep \mu\ll\lambda\}$ is isometric to $L^1([0,1])$ (by the Radon-Nikod\'ym theorem), we deduce that $Y$ is separable. Hence $\Ienv(Y)=Y$.

On the other hand, $\lambda|_E\in\overline{Y}^{\sigma(\M([0,1]),\B^b_1([0,1]))}\setminus Y$. Indeed, assume that $\lambda|_E\notin\overline{Y}^{\sigma(\M([0,1]),\B^b_1([0,1]))}$. Then, by the Hahn-Banach theorem, there is a function $g\in \B^b_1([0,1])$ such that
$$\ip{\lambda|_E}{g}\ne 0 \ \&\ \forall\mu\in Y:\ip\mu g=0.$$
The second condition implies
$$\int fg\di\lambda =0 \mbox{ whenever }f\in L^1([0,1]), f|_E=0,$$
hence $g=0$ almost everywhere on $[0,1]\setminus E$.
The first condition means $\int_E g\di\lambda\ne 0$.

Assume now that $A,B\subset E$ are two Borel sets of positive Lebesgue measure. Then 
$$\frac1{\lambda(A)}\lambda|_A-\frac1{\lambda(B)}\lambda|_B \in Y,$$
hence
$$\frac1{\lambda(A)}\int_A g\di\lambda=\frac1{\lambda(B)}\int_B g\di\lambda.$$
It follows that $g$ is essentially constant on $E$, i.e., there is a constant $c$ such that $g=c$ almost everywhere on $E$. Necessarily $c\ne0$. 

It follows that $g=0$ on a dense set and $g=c$ on a dense set as well. Hence $g$ has no point of continuity, so it cannot be of the first Baire class. This contradiction completes the proof.
\end{proof}

We finish this section by two open problems. The first one concerns possible converse to Theorem~\ref{T:X not containing ell1}.

\begin{ques}\label{q:ell1}
Assume that $X$ is a Banach space containing an isomorphic copy of $\ell^1$. Does there exist a convex subset $A\subset X^*$ such that
$$\Ienv{A}\subsetneqq \overline{A}^{\sigma(X^*,B_1(X))}\ ?$$
In particular, is it true for $X=\C(K)$, where $K$ is any compact space which is not scattered?
\end{ques}

We know that the answer is positive if $X$ contains a complemented copy of $\ell^1$, if $X$ is weakly sequentially complete or for $X=\C(K)$ where $K$ is an uncountable metrizable compact space.

Another question is what happens if we restrict our attention just to bounded sets. Note that the original motivation for introducing (I)-envelopes comes from (I)-generation used in \cite{FL} to investigate James boundaries of weak$^*$ compact convex sets, so in this case the (I)-envelope is applied to a bounded set. The study of (I)-envelopes of unbounded sets, in particular of subspaces, is also natural (see e.g., \cite[Theorem 2.1]{Ienv-studia} where it is used to provide an easy characterization of Grothendieck spaces), however the main focus is on bounded sets. Therefore the following question seems to be natural.

\begin{ques}\label{q:bounded}
Let $X$ be a Banach space. Is there a locally convex topology $\tau$ on $X^*$ such that $\Ienv(A)=\overline{\co A}^\tau$ for each bounded $A\subset X^*$?
\end{ques}

Note that the counterexamples in Example~\ref{ex:ell1} and Theorem~\ref{T:CK} are based on Corollary~\ref{cor:uniqueness} which follows from Lemma~\ref{L:Ienv for linear subspace} describing the (I)-envelope of subspaces (which are, of course, unbounded). So, our methods do not help to solve Question~\ref{q:bounded}.

\section{Several intermediate topologies}\label{sec:topologies}

In the previous section we analyzed the problem of existence of a locally convex topology which may be used to describe the (I)-envelopes. We saw that in some cases it exists and in some cases it does not and that the full characterization is still missing.  

Next we will look at the problem of existence any such topology, not necessarily locally convex. This requires a completely different approach. We will define several natural closure operators and analyze their mutual relationships and connections to the (I)-envelope.

We start by recalling what is a closure operator. Let $X$ be a nonempty set. By a \emph{closure operator} on $X$ we understand a mapping $\gamma:\P(X)\to\P(X)$ (note that $\P(X)$ denotes the power set of $X$) with the following properties (cf. \cite[Remark on p. 7]{LMZ}).
\begin{enumerate}[(i)]
    \item $\gamma(\emptyset)=\emptyset$;
    \item $\gamma(A)\supset A$ for each $A\subset X$;
    \item $\gamma(A\cup B)=\gamma(A)\cup\gamma(B)$ for $A,B\subset X$.
\end{enumerate}
If $\gamma$ satisfies moreover
\begin{enumerate}[(i{v})]
    \item $\gamma(\gamma(A))=\gamma(A)$ for each $A\subset X$;
\end{enumerate}
it is called \emph{idempotent closure operator}.

If $\gamma$ is a closure operator, we call a set $A\subset X$  \emph{$\gamma$-closed} if $\gamma(A)=A$. It is well known and easy to see that the collection of all $\gamma$-closed sets is the collection of all closed sets in a topology on $X$. If $\gamma$ is moreover idempotent, then $\gamma(A)$ is the closure of $A$ in this topology (cf. \cite[Proposition 1.2.7]{engelking}).

Next we present an abstract lemma on properties of `intermediated closure operators'.

\begin{lemma}\label{L:abstract Icl}
Let $X$ be a nonempty set, $\beta:\P(X)\to\P(X)$ a closure operator and $\alpha:\P(X)\to\P(X)$ a monotone mapping satisfying conditions (i) and (ii) from the definition of a closure operator and, moreover, $\alpha(A)\supset\beta(A)$ for each $A\subset X$.
Let us set
$$\begin{aligned}
\gamma_1(A)&=\bigcap\left\{\beta\left(\bigcup_{k=1}^n \alpha(A_k)\right)\setsep A=\bigcup_{k=1}^n A_k, n\in\en
\right\},\\
\gamma_2(A)&=\bigcap\left\{\beta\left(\bigcup_{n=1}^\infty \alpha(A_n)\right)\setsep A=\bigcup_{n=1}^\infty A_n
\right\}.
\end{aligned}$$
Then the following assertions hold:
\begin{enumerate}[\rm(a)]
    \item Both $\gamma_1$ and $\gamma_2$ are closure operators.
    \item $\beta(A)\subset\gamma_2(A)\subset\gamma_1(A)\subset\alpha(\alpha(A))$ for $A\subset X$.
    \item If $\alpha$ is idempotent, then $\gamma_1$ is idempotent as well.
    \item Assume that:
    \begin{itemize}
        \item $X$ is a topological vector space and $\beta$ is the closure in the toplogy of $X$;
        \item $\alpha$ is idempotent;
        \item there is $\U$, a base of neighborhoods of zero in $X$ formed by balanced sets such that $\alpha(\alpha(A)+U)=\overline{\alpha(A)+U}$ for any $A\subset X$ and any $U\in\U$.
    \end{itemize}
    Then $\gamma_2$ is idempotent as well and, moreover,
$$\beta(A)\subset\gamma_2(A)\subset\gamma_1(A)\subset\alpha(A)\mbox{ for }A\subset X.$$
\end{enumerate}
\end{lemma}

Recall that a subset $A$ of a vector space is \emph{balanced} if $\alpha A\subset A$ whenever $\abs{\alpha}\le1$.

\begin{proof} (a) It is clear that both $\gamma_1$ and $\gamma_2$ satisfy properties (i) and (ii). It is further easy to check that that both mappings are monotone. Therefore,  inclusion `$\supset$' in (iii) follows. Let us prove the converse one.

Assume that $x\in X\setminus (\gamma_1(A)\cup\gamma_1(B))$. Then there are coverings
 $$A=A_1\cup\dots\cup A_n, B=B_1\cup\dots\cup B_m$$ 
 such that
 $$x\notin \beta\left(\bigcup_{j=1}^n \alpha(A_j)\right)\cup \beta\left(\bigcup_{j=1}^m \alpha(B_j)\right)= \beta\left(\bigcup_{j=1}^n \alpha(A_j)\cup \bigcup_{j=1}^m \alpha(B_j)\right),$$
 so clearly $x\notin\gamma_1(A\cup B)$. The proof for $\gamma_2$ is completely analogous. Hence $\gamma_1$ and $\gamma_2$ are indeed closure operators.

(b) The first inclusion follows from the monotonicity of $\beta$. The second one is easy (note that a finite cover may be extended to an infinite one by adding empty sets and $\alpha(\emptyset)=\emptyset$). The last follows by taking the cover of $A$ by just one set -- $A$ -- hence $\gamma_1(A)\subset\beta(\alpha(A))\subset\alpha(\alpha(A))$.

(c) Fix $A\subset X$. Clearly $\gamma_1(\gamma_1(A))\supset\gamma_1(A)$. To prove the converse, pick any $x\notin \gamma_1(A)$. Then there is a cover $A=\bigcup_{k=1}^n A_k$ such that
$$x\notin\beta\left(\bigcup_{k=1}^n \alpha(A_k)\right).$$
Then
$$\gamma_1(A)\subset\beta\left(\bigcup_{k=1}^n \alpha(A_k)\right)=
\bigcup_{k=1}^n \beta(\alpha(A_k))\subset\bigcup_{k=1}^n \alpha(\alpha(A_k))=\bigcup_{k=1}^n \alpha(A_k),$$
hence
$$\gamma_1(\gamma_1(A))\subset\beta\left(\bigcup_{k=1}^n \alpha(\alpha(A_k))\right)=\beta\left(\bigcup_{k=1}^n \alpha(A_k)\right),$$
so $x\notin\gamma_1(\gamma_1(A))$.

(d) Fix $A\subset X$. Clearly $\gamma_2(\gamma_2(A))\supset\gamma_2(A)$. To prove the converse, pick any $x\notin \gamma_2(A)$. Then there is a cover $A=\bigcup_{n=1}^\infty A_n$ such that
$$x\notin\overline{\bigcup_{n=1}^\infty \alpha(A_n)}.$$
Then there $U$, a  neighborhood of zero in $X$, such that
$$(x+U)\cap {\bigcup_{n=1}^\infty \alpha(A_n)}=\emptyset.$$
Let $V\in\U$ be such that $V+V\subset U$. Then
$$(x+V)\cap {\bigcup_{n=1}^\infty (\alpha(A_n)+V)}=\emptyset.$$
Since 
$$\gamma_2(A)\subset \overline{\bigcup_{n=1}^\infty \alpha(A_n)}\subset\bigcup_{n=1}^\infty (\alpha(A_n)+V),$$
we get
$$\gamma_2(\gamma_2(A))\subset\overline{\bigcup_{n=1}^\infty \alpha(\alpha(A_n)+V)}\subset
\overline{\bigcup_{n=1}^\infty \overline{\alpha(A_n)+V}}
=\overline{\bigcup_{n=1}^\infty (\alpha(A_n)+V)}.$$
Hence, $x\notin \gamma_2(\gamma_2(A))$.

The `moreover' part follows from (b).
\end{proof}

Next we going to define several concrete intermediate closure operators. Their properties will be collected and compared later. Assume that $X$ is a Banach space and $A\subset X^*$.

We start by the following pair of operators:
$$\begin{aligned}
\Icl(A)&= \bigcap\left\{\overline{\bigcup_{n=1}^\infty \wscl{A_n}}^{\norm{\cdot}}\setsep A=\bigcup_{n=1}^\infty A_n
\right\},\\\Iccl(A)=& \bigcap\left\{\overline{\bigcup_{n=1}^\infty \wscl{\co A_n}}^{\norm{\cdot}}\setsep A=\bigcup_{n=1}^\infty A_n
\right\}.\end{aligned}$$
These two operators are inspired by the definition of (I)-envelope 
-- in the first case we just omit the convex hulls in the formula, in the second case we omit the outer convex hull. There are further variants of these operators inspired by the Mazur theorem: In the definition of the (I)-envelope we use the weak$^*$-closure and the norm-closure. Since we apply these closures to convex sets, it does not matter which of the topologies with the same dual we use. However, if we apply the respective closures to non-convex sets, the results may differ. There are many possibilities, we point out two extreme ones: 
$$\begin{aligned}
\Imn(A)&= \bigcap\left\{\overline{\bigcup_{n=1}^\infty \overline{A_n}^{\mu(X^*,X)}}^{\norm{\cdot}}\setsep A=\bigcup_{n=1}^\infty A_n
\right\},\\\Iw(A)=& \bigcap\left\{\overline{\bigcup_{n=1}^\infty \wscl{\co A_n}}^{w}\setsep A=\bigcup_{n=1}^\infty A_n
\right\}.\end{aligned}$$
Note that $\mu(X^*,X)$ denotes the respective Mackey topology (cf. \cite[Definition 3.43 and Corollary 3.44]{Fab-BST}).

Next we define two more closure operators, this time using the notion of (I)-envelope. For $A\subset X^*$ we set
$$\begin{aligned}
\cl_{IF}(A)&=\bigcap\left\{\bigcup_{j=1}^n \Ienv(A_j)\setsep A=\bigcup_{j=1}^n A_j,n\in\en\right\},\\
\cl_{IC}(A)&=\bigcap\left\{\overline{\bigcup_{n=1}^\infty \Ienv(A_n)}^{\norm{\cdot}}\setsep A=\bigcup_{n=1}^\infty A_n\right\}.
\end{aligned}$$

We finally introduce one more pair of closure operators, inspired by Proposition~\ref{P:Ienv between closures}. We again choose two extremes from the possibilities of the topologies with prescribed dual. For $A\subset X^*$ we set
$$\begin{aligned}
\Ibl(A)&= \bigcap\left\{\overline{\bigcup_{n=1}^\infty \overline{A_n}^{\mu(X^*,B_1(X))}}^{\mu(X^*,C(X))}\setsep A=\bigcup_{n=1}^\infty A_n
\right\},\\
\Ibcl(A)&= \bigcap\left\{\overline{\bigcup_{n=1}^\infty \overline{\co A_n}^{\sigma(X^*,B_1(X))}}^{\sigma(X^*,C(X))}\setsep A=\bigcup_{n=1}^\infty A_n
\right\}.\end{aligned}$$

Next we establish basic properties of the above-defined operators.

\begin{prop}\label{P:eight operators} Let $X$ be a Banach space.
\begin{enumerate}[\rm(a)]
    \item  The eight above-defined mappings are closure operators on $X^*$.
    \item The operators $\Icl,\Iccl,\cl_{IF}$ are idempotent.
    \item Let $\gamma$ be a closure operator on $X^*$ such that $\gamma(A)\subset\Ienv(A)$ for each $A\subset X^*$. Then $\gamma(A)\subset \cl_{IF}(A)$ for each $A\subset X^*$.
\end{enumerate}
\end{prop}

\begin{proof} Note that $\cl_{IF}$ is of the form $\gamma_1$ and the remaining six mappings are of the form $\gamma_2$ from Lemma~\ref{L:abstract Icl} for suitable choices of $\alpha$ and $\beta$. Let us review these choices in the individual cases.

\begin{description}
\item[$\Icl$] $\beta(A)=\overline{A}^{\norm{\cdot}}$, $\alpha(A)=\wscl{A}$;
\item[$\Iccl$] $\beta(A)=\overline{A}^{\norm{\cdot}}$, $\alpha(A)=\wscl{\co A}$;
\item[$\Imn$] $\beta(A)=\overline{A}^{\norm{\cdot}}$, $\alpha(A)=\overline{A}^{\mu(X^*,X)}$;
\item[$\Iw$] $\beta(A)=\overline{A}^w$, $\alpha(A)=\wscl{\co A}$;
\item[$\cl_{IF}$] 
$\beta(A)=\overline{A}^{\norm{\cdot}}$, $\alpha(A)=\Ienv(A)$;
\item[$\cl_{IC}$] 
$\beta(A)=\overline{A}^{\norm{\cdot}}$, $\alpha(A)=\Ienv(A)$;
\item[$\Ibl$] $\beta(A)=\overline{A}^{\mu(X^*,C(X))}$, $\alpha(A)=\overline{A}^{\mu(X^*,B_1(X))}$;
\item[$\Ibcl$] $\beta(A)=\overline{A}^{\sigma(X^*,C(X))}$, $\alpha(A)=\overline{\co A}^{\sigma(X^*,B_1(X))}$.
\end{description}

In all cases $\beta$ is a closure operator and $\alpha$ has the required properties. So, assertion (a) is proved.

Let us continue by proving (b). For $\cl_{IF}$ we may use Lemma~\ref{L:abstract Icl}(c). In the remaining cases we will use Lemma~\ref{L:abstract Icl}(d). Indeed, for $\Icl$ and $\Iccl$ the operator $\beta$ is the norm closure and $\alpha$ is in both cases idempotent. 
It remains to prove that in these cases we have
$$\alpha(\alpha(A)+U(0,r))=\overline{\alpha(A)+U(0,r)}^{\norm{\cdot}}$$
for $A\subset X^*$.
So, fix $A\subset X^*$. Then
$$\wscl{\wscl{A}+U(0,r)}\subset \wscl{A}+\overline{U(0,r)}\subset \overline{\wscl{A}+U(0,r)}^{\norm{\cdot}}$$
as $\overline{U(0,r)}$ is weak$^*$-compact by the Banach-Alaoglu theorem and hence $\wscl{A}+\overline{U(0,r)}$ is weak$^*$-closed. The second case is similar, we only use that the sum of two convex sets is again convex.

(c) Let $A=\bigcup_{k=1}^n A_k$ Then
$$\gamma(A)=\bigcup_{k=1}^n\gamma(A_k)\subset\bigcup_{k=1}^n \Ienv(A_k).$$
By taking intersection over all finite covers of $A$ we get $\gamma(A)\subset\cl_{IF}(A)$.
\end{proof}

We continue by the following theorem which characterizes Banach spaces in which the (I)-envelope can be described using a closure in a topology. Let us point out that it is still not clear whether this condition is satisfied by any Banach space.

\begin{thm}\label{T:characterization} Let $X$ be a Banach space. The following assertions are equivalent.
\begin{enumerate}[\rm(1)]
    \item There is a topology $\tau$ on $X^*$ such that $\Ienv(A)=\overline{\co A}^\tau$ for each $A\subset X^*$.
    \item If $A_1,\dots,A_n$ are convex subsets of $X^*$ such that $A_1\cup\dots\cup A_n$ is also convex, then 
    $$\Ienv(A_1\cup\dots\cup A_n)=\Ienv(A_1)\cup\dots\cup\Ienv(A_n).$$
    \item $\Ienv(A)=\cl_{IF}(\co A)$ for each $A\subset X^*$.
\end{enumerate}

\end{thm}

\begin{proof}
$(1)\Rightarrow(2)$: Let $\tau$ be the topology provided by (1). Then $\Ienv(A)=\overline{A}^\tau$ for each convex set $A\subset X^*$. Since $$\overline{A_1\cup\dots\cup A_n}^\tau=\overline{A_1}^\tau\cup\dots\cup\overline{A_n}^\tau,$$
assertion (2) easily follows.

$(2)\Rightarrow(3)$: Assume (2) holds. Since $\Ienv(A)=\Ienv(\co A)$, it is enough to prove the equality from (3) in case $A$ is convex. So, fix a convex set $A\subset X^*$. Fix $A_1,\dots,A_n$ such that $A=A_1\cup\dots\cup A_n$. Then
$$\bigcup_{j=1}^n\Ienv(A_j)=\bigcup_{j=1}^n\Ienv(\co A_j)=\Ienv\left(\bigcup_{j=1}^n\co A_j\right)=\Ienv(A),$$
where we used assumption (2) and the convexity of $A$. Now it follows that $\cl_{IF}(A)=\Ienv(A)$ and the proof is complete.

$(3)\Rightarrow(2):$ This follows from the fact that $\cl_{IF}$ is an idempotent closure operator.
\end{proof}

The previous theorem says, in particular, that to answer the key question of our paper it would be enough to understand the closure operator $\cl_{IF}$. However, this closure operator has one disadvantage -- it is defined using (I)-envelope, which makes the understanding more difficult. Most of the remaining ones are defined without referring to (I)-envelopes, using only suitable
combinations of standard topologies on a dual Banach space. Therefore, we find interesting to study their mutual relationships.
It appears that these relationships depend on (geometrical or topological) properties of the given Banach space. Even more important for our problem is to compare them for convex sets
(see Proposition~\ref{P:convex comparison} and the subsequent examples). 

We start by summarizing inclusions which hold in general and in some important special cases.

\begin{prop}\label{P:comparison of closure operators}
Let $X$ be a Banach space and let $A\subset X^*$ be arbitrary.

\begin{enumerate}[\rm(a)]
    \item The following inclusions hold:
    $$\begin{array}{ccccccccc}
        \Icl(A)&=&\Iccl(A)&\subset&\Iw(A)&\subset&\cl_{IF}(A)&\subset& \Ienv(A)\cap \wscl{A}  \\
        \bigcup &&\bigcup&&&&&&\bigcap \\
        \Imn(A) &&\cl_{IC}(A)&&&&  &&\wscl{A} \\
        \bigcup &\sesub&&&&&&&\bigcup \\
        \overline{A}^{\norm{\cdot}}&\subset&\overline{A}^{\mu(X^*,C(X))}&\subset&\Ibl(A)&\subset&\Ibcl(A)&\subset& \overline{A}^{\sigma(X^*,B_1(X))}
    \end{array}
    $$
    \item Assume $A$ is norm-separable. Then
    $$\Iccl(A)=\overline{A}^{\norm{\cdot}}, \Iw(A)=\cl_{IF}(A)=\overline{A}^{w}, 
    \Ienv(A)=\overline{\co A}^{\norm{\cdot}}.$$
    \item Assume $X$ is separable. Then
    $$\overline{A}^{\norm{\cdot}}=\overline{A}^{\mu(X^*,C(X))}, \Ibl(A)\subset \Imn(A)\mbox{ and }\Ibcl(A)\subset\Iw(A).$$
    \item Assume $X$ is a separable Banach space not containing an isomorphic copy of $\ell_1$. Then
    $$ \Ienv(A)=\overline{\co A}^{\norm{\cdot}}, \Ibcl(A)=\Iw(A)=\cl_{IF}(A)=\overline{A}^{w}, 
    $$
\end{enumerate}
\end{prop}

\begin{proof}
(a) We start by the first equality on the first row. Inclusion `$\subset$' is obvious. To prove the converse take any $x^*\in X^*\setminus \Icl(A)$. It means that there is a sequence $(A_n)$ of sets covering $A$ such that
$$x^*\notin \overline{\bigcup_{n\in\en}\wscl{A_n}}.$$
I.e., there is some $r>0$ such that
$$x^*+rB_{X^*}\cap \bigcup_{n\in\en}\wscl{A_n}=\emptyset.$$
For each $y^*\in\bigcup_{n\in\en}\wscl{A_n}$ there are, by the Hahn-Banach theorem, some $x_{y^*}\in X$ and $c_{y^*}\in\er$ such that
$$\forall z^*\in B_{X^*}\colon \Re y^*(x_{y^*})<c_{y^*}\le x^*(x_{y^*})+r z^*(x_{y^*}).$$
Set
$$H_{y^*}=\{z^*\in X^*\setsep z^*(x_{y^*})<c_{y^*}\}.$$
Then $H_{y^*}$ is a weak$^*$-open half-space disjoint with $z^*+rB_{X^*}$ and containing $y^*$.
Hence the system 
$$H_{y^*},\quad y^*\in \bigcup_{n\in\en}\wscl{A_n}$$
is a cover of $\bigcup_{n\in\en}\wscl{A_n}$ by weak$^*$-open sets. The set $\bigcup_{n\in\en}\wscl{A_n}$, being weak$^*$-$F_\sigma$ and hence $\sigma$-compact in the weak$^*$-topology, is weak$^*$-Lindel\"of. It follows that there is a countable subcover. I.e., there is a sequence $(y_k^*)$ such that
$$\bigcup_{n\in\en}\wscl{A_n}\subset\bigcup_{k\in\en} H_{y_k^*}.$$
We deduce that
$$\bigcup_{k\in\en} \wscl{H_{y_k^*}}$$
is a cover of $A$ by convex weak$^*$-closed sets disjoint with the norm interior of the ball $x^*+rB_{X^*}$. Hence $x^*\notin\Iccl(A)$, which completes the proof of inclusion `$\supset$' and hence of the equality.

Continuing the first row, the second  inclusion is obvious. The third one follows from Proposition~\ref{P:eight operators}(c) since $\Iw$ is a closure operator and clearly $\Iw(A)\subset \Ienv(A)$ for $A\subset X^*$.

Let us prove the fourth inclusion.  The first part, i.e., the inclusion $\cl_{IF}(A)\subset \Ienv(A)$ is obvious. Let us prove that $\cl_{IF}(A)\subset\wscl A$.
 
 Assume that $x^*\notin\wscl{A}$. Then there are $x_1,\dots,x_n\in X$ and $\varepsilon>0$ such that
 $$\{y^*\in X^*\setsep \abs{\Re(y^*(x_j)-x^*(x_j))}<\varepsilon \mbox{ for }j=1,\dots,n\}\cap A=\emptyset.$$
 For $j=1,\dots,n$ set
 $$A_j^+=\{y^*\in X^*\setsep \Re y^*(x_j)\ge \Re x^*(x_j)+\varepsilon\} \mbox{ and }
 A_j^-=\{y^*\in X^*\setsep \Re y^*(x_j)\le \Re x^*(x_j)-\varepsilon\}.$$
 Then
 $$A\subset \bigcup_{j=1}^n(A_j^+\cup A_j^-)$$
 and, moreover, the sets $A_j^+$ and $A_j^-$ are convex and weak$^*$-closed.
 Hence $\cl_{IF}(A)\subset\bigcup_{j=1}^n(A_j^+\cup A_j^-)$ and therefore $x^*\notin\cl_{IF}(A)$.
 
 All vertical inclusions are obvious and the first three inclusions on the last line as well.
 
 The last inclusion, $\Ibcl(A)\subset\overline{A}^{\sigma(X^*,B_1(X))}$ may be proved in the same way as the inclusion $\cl_{IF}(A)\subset\wscl{A}$ (we just use elements of $B_1(X)$ instead of $X$).
 
 (b) Assume $A$ is norm-separable. Then $\Ienv(A)=\overline{\co A}^{\norm{\cdot}}$ by \cite[Remark 1.1(ii)]{Ienv-israel}. Using the same argument coming from \cite[Proposition 2.2(a)]{FL} we may show that $\Iccl(A)=\overline{A}^{\norm{\cdot}}$. Let us give the argument for the sake of completeness.
 
 Let $(x_n^*)$ be a norm-dense sequence in $A$. Let $\varepsilon>0$ be arbitrary. Then $A\subset \bigcup_n (x_n^*+\varepsilon B_{X^*})$. Since the set $x_n^*+\varepsilon B_{X^*}$ is convex and weak$^*$-closed for each $n$, we deduce that
 $$\Iccl(A)\subset \overline{\bigcup_n (x_n^*+\varepsilon B_{X^*})}^{\norm{\cdot}}\subset A+2\varepsilon B_{X^*}.$$
 Since $\varepsilon>0$ is arbitrary, we deduce that $\Iccl(A)\subset\overline{A}^{\norm{\cdot}}$. The converse inclusion follows from (a).
 
 Further, we have
 $$\overline{A}^w\subset \Iw(A)\subset\cl_{IF}(A).$$
 Indeed, the first inclusion is obvious and the second one follows from (a). To prove the converse inclusions fix $x^*\notin \overline{A}^w$. Then there are $x^{**}_1,\dots,x^{**}_n\in X^{**}$ and $\varepsilon>0$ such that
 $$\{y^*\in X^*\setsep \abs{\Re x_j^{**}(y^*-x^*)}<\varepsilon \mbox{ for }j=1,\dots,n\}\cap A=\emptyset.$$
 For $j=1,\dots,n$ set
 $$A_j^+=\{y^*\in X^*\setsep \Re x_j^{**}(y^*-x^*)\ge \varepsilon\} \mbox{ and }
 A_j^-=\{y^*\in X^*\setsep \Re x_j^{**}(y^*-x^*)\le -\varepsilon\}.$$
 Then
 $$A\subset \bigcup_{j=1}^n(A_j^+\cup A_j^-)$$
 and, moreover, the sets $A_j^+$ and $A_j^-$ are convex and norm-closed.
 Hence 
 $$\begin{aligned}
      \cl_{IF}(A)&\subset\bigcup_{j=1}^n(\Ienv(A_j^+\cap A)\cup \Ienv(A_j^-\cap A))=\bigcup_{j=1}^n\left(\overline{\co (A_j^+\cap A)}^{\norm{\cdot}}\cup\; \overline{\co(A_j^-\cap A)}^{\norm{\cdot}}\right)\\&\subset  \bigcup_{j=1}^n(A_j^+\cup A_j^-)\end{aligned}$$ and therefore $x^*\notin\cl_{IF}(A)$. 
 
 (c) If $X$ is separable, then $C(X)=X^{**}$. Hence the topology $\sigma(X^*,C(X))$ coincides with the weak topology and the topology $\mu(X^*,C(X))$ coincides with the norm topology. Hence the statement easily follows.
 
 (d) If $X$ is separable and does not contain a copy of $\ell_1$, then $B_1(X)=C(X)=X^{**}$ (by \cite{odell-rosenthal}), hence the topologies $\sigma(X^*,B_1(X))$ and $\sigma(X^*,C(X))$ coincide with the weak topology on $X^*$ and the topologies $\mu(X^*,B_1(X))$ and $\mu(X^*,C(X))$ coincide with the norm topology on $X^*$.
 
 The equality $\Ienv(A)=\overline{\co A}^{\norm{\cdot}}$ then follows (for example) from Theorem~\ref{T:X not containing ell1}. 
 
 The equalities $\Iw(A)=\cl_{IF}(A)=\overline{A}^{w}$ may be proved by copying the argument from (b). Furhter, clearly $\overline{A}^w\subset \Ibcl(A)\subset\Iw(A)$, hence $\Ibcl(A)=\overline{A}^w$.
  \end{proof}

\begin{cor}\label{cor:separable case} Let $X$ be a separable Banach space and $A\subset X^*$ a norm-separable subset. Then
$$\overline{A}^{\norm{\cdot}}=\Iccl(A)=\Ibl(A)\subset \overline{A}^w=\Iw(A)=\cl_{IF}(A)=\Ibcl(A).$$
If $X$ does not contain a copy of $\ell_1$, then additionally  $\Ibcl(A)=\overline{A}^{\sigma(X^*,B_1(X))}$.
\end{cor}

\begin{proof}
The first two equalities follow from Proposition~\ref{P:comparison of closure operators}(a)--(c). The equalities $\overline{A}^w=\Iw(A)=\cl_{IF}(A)$ follow from Proposition~\ref{P:comparison of closure operators}(b). By Proposition~\ref{P:comparison of closure operators}(c) we get $\Ibcl(A)\subset\cl_{IF}(A)$. The inclusion $\overline{A}^w\subset \Ibcl(A)$ is obvious. Finally, $\Iccl(A)\subset\Iw(A)$ by Proposition~\ref{P:comparison of closure operators}(a).

If $X$ does not contain a copy of $\ell_1$, then $B_1(X)=X^{**}$, hence $\overline{A}^{w}=\overline{A}^{\sigma(X^*,B_1(X))}$.
\end{proof}

Next we focus on distinguishing the individual closure operators. At first we look at the case of a norm-separable subset of the dual to a separable Banach space. The next example witnesses that in such a case no more inclusions hold, besides those from Proposition~\ref{P:comparison of closure operators}(a) completed by Corollary~\ref{cor:separable case}.

\begin{example2}
(1) Let $X$ be an infinite dimensional separable reflexive space  (for example  $X=\ell_2$). Then  $X^{**}=C(X)=B_1(X)=X$. Hence the topologies $\sigma(X^*,B_1(X))$ and $\sigma(X^*,C(X))$ coincide with the weak (and also with the weak$^*$) topology. Futher, topologies $\mu(X^*,C(X))$ and $\mu(X^*,B_1(X))$ coincide with the norm topology.
 Let $A=S_{X^*}$, the unit sphere of $X^*$. Using Corollary~\ref{cor:separable case} we get
$$\begin{gathered}\Iccl(A)=\Ibl(A)=\overline{A}^{\norm{\cdot}}=A \ (=S_{X^*}),\\
\Iw(A)=\cl_{IF}(A)=\Ibcl(A)=\wscl{A}=\Ienv(A)=B_{X^*}.
\end{gathered}$$
In particular,
$$\Iccl(A)\subsetneqq\Iw(A)\mbox{ and }\Ibl(A)\subsetneqq\Ibcl(A),$$
and, further
$$\Ibcl(A)\not\subset \Iccl(A)\mbox{ and }\Iw(A)\not\subset\Ibl(A).$$

(2) Let $X$ be an infinite dimensional nonreflexive space with  separable dual (for example $X=c_0$). Then  $X^{**}=C(X)=B_1(X)$. Hence the topologies $\sigma(X^*,B_1(X))$ and $\sigma(X^*,C(X))$ coincide with the weak (but not with the weak$^*$) topology. Further, topologies $\mu(X^*,C(X))$ and $\mu(X^*,B_1(X))$ coincide with the norm topology. 

If $A=S_{X^*}$, the unit sphere of $X^*$, we get the same equalities as in (1).

Further, let $\varphi\in X^{**}\setminus X$ and let $Y=\ker \varphi$. Set $B=S_Y$, the unit sphere of $Y$. Then we get
$$\begin{gathered}\Iccl(B)=\Ibl(B)=\overline{B}^{\norm{\cdot}}=B \ (=S_{Y}),\\
\Iw(B)=\cl_{IF}(B)=\Ibcl(B)=\overline{B}^w=\Ienv(B)=\overline{B}^{\sigma(X^*,B_1(X))}=B_{Y},\\
\wscl{B}\supsetneqq B_Y.
\end{gathered}$$
Indeed, the equalities on the first two lines follow from Corollary~\ref{cor:separable case}, the inclusion on the third line follows from the fact that $Y$ is not weak$^*$ closed (and the same holds for $B$).

In particular,
$$\Ienv(B)\cap\wscl{B}\subsetneqq \wscl{B} \mbox{ and }\overline{B}^{\sigma(X^*,B_1(X))}\subsetneqq\wscl{B}.$$

(3) Let $X=J^*$, where $J$ is the James space (see \cite{james} or \cite[Section 3.4]{Albiac-Kalton}). Again $X^{**}=C(X)=B_1(X)$, the topologies $\sigma(X^*,B_1(X))$ and $\sigma(X^*,C(X))$ coincide with the weak topology and the topologies $\mu(X^*,C(X))$ and $\mu(X^*,B_1(X))$ coincide with the norm topology. If $A=S_J\subset J^{**}=X^*$, then
$$\begin{gathered}\Iccl(A)=\Ibl(A)=\overline{A}^{\norm{\cdot}}=A \ (=S_{J}),\\
\Iw(A)=\cl_{IF}(A)=\Ibcl(A)=\overline{A}^w=\Ienv(A)=\overline{A}^{\sigma(X^*,B_1(X))}=B_{J},\\
\wscl{A}=B_{X^*}=B_{J^{**}}\supsetneqq B_J.
\end{gathered}$$
This follows similarly as (2) above, using moreover the Goldstine theorem.

(4) Let $X=\ell_1$. Then $C(X)=X^{**}$ and $B_1(X)=X$. Hence the topology  $\sigma(X^*,C(X))$ coincides with the weak topology and $\mu(X^*,C(X))$ coincides with the norm topology.
Further, $\sigma(X^*,B_1(X))$ coincides with the weak$^*$ topology. The topology $\mu(X^*,B_1(X))$ is the topology of uniform convergence on weakly compact subsets of $X$ (by the Mackey-Arens theorem, cf. \cite[Corollary 3.44]{Fab-BST}). Since $X=\ell_1$ has the Schur property (see, e.g., \cite[Theorem 2.3.6]{Albiac-Kalton}), it further coincides with the topology of uniform convergence on norm-compact subspace of $X$, so on bounded sets it coincides with the weak$^*$-topology. 

Let $A=S_{c_0}$, the unit sphere of $c_0$ canonically embedded into $X^*=\ell_\infty$. Then we get
$$\begin{gathered}\Iccl(A)=\Ibl(A)=\overline{A}^{\norm{\cdot}}=A \ (=S_{c_0}),\\
\Iw(A)=\cl_{IF}(A)=\Ibcl(A)=\overline{A}^w=\Ienv(A)=B_{c_0},\\
\overline{A}^{\sigma(X^*,B_1(X))}=\wscl{A}=B_{X^*}=B_{\ell_\infty}.
\end{gathered}$$
Again, the first two lines follow from Corollary~\ref{cor:separable case} and the third line follows from the Goldstine theorem.

In particular,
$$\Ibcl(A)\subsetneqq\overline{A}^{\sigma(X^*,B_1(X))}\mbox{ and }\overline{A}^{\sigma(X^*,B_1(X))}\not\subset\Ienv(A).$$

(5) Let $X=c_0$ and $A=\{\pm e_n^*\setsep n\in\en\}$, where $e_n^*$ are the canonical basic vectors in $\ell_1=c_0^*$. Then $A$ is weakly closed, $\wscl{A}=A\cup\{0\}$ and $\Ienv(A)=\overline{\co A}^{\norm{\cdot}}=B_{X^*}$.
In particular, by Corollary~\ref{cor:separable case} we deduce that
$$\cl_{IF}(A)\subsetneqq\Ienv(A)\cap \wscl{A}.$$
Further, since $B_1(X)=X^{**}$, $A$ is also $\sigma(X^*,B_1(X))$-closed, we get
$$\Ienv(A)\cap \wscl{A}\not\subset \overline{A}^{\sigma(X^*,B_1(X))}.$$
\end{example2}

We continue by distinguishing the closure operators in the general case, without assuming separability. In this case we are able to distinguish more operators (see the following example) but we do not know the exact picture (cf. Question~\ref{q:remaining}
below).

\begin{example2}\label{exa:nonseparable}
(1) Let $X=\ell_1(\Gamma)$ for an uncountable set $\Gamma$ and $A=\{e_\gamma^*\setsep \gamma\in \Gamma\}$, the canonical unit vectors in $X^*=\ell_\infty(\Gamma)$. Then $A$ is norm-closed.

Since $X$ has the Schur property, weakly compact subsets in $X$ are norm-compact. It follows that the topology $\mu(X^*,X)$ coincides with the weak$^*$-topology on bounded sets. Thus, $0\in\overline{B}^{\mu(X^*,X)}$ for any infinite $B\subset A$. It follows that $0\in\Imn(A)$, hence
$$\overline{A}^{\norm{\cdot}}\subsetneqq \Imn(A).$$
Similarly we deduce that $0\in\Ienv(B)$ for any uncountable $B\subset A$. Therefore, $0\in\cl_{IC}(A)$, so
$$\overline{A}^{\norm{\cdot}}\subsetneqq \cl_{IC}(A).$$

Further, since $X$ is weakly sequentially complete, we have $B_1(X)=X$, so $\mu(X^*,B_1(X))=\mu(X^*,X)$. Hence, similarly as in the previous paragraph we see that $0\in\Ibl(A)$. It follows that
$$\overline{A}^{\norm{\cdot}}\subsetneqq \Ibl(A).$$

%{\color{purple} Question: Is $C(X)=X^{**}$ in this case? Probably not.}

(2) Let $X=\ell_2(\Gamma)$ for an uncountable set $\Gamma$ and $A=\{e_\gamma^*\setsep \gamma\in \Gamma\}$, the canonical unit vectors in $X^*=\ell_2(\Gamma)$. Then $A$ is norm-closed.

Since $X$ is reflexive, we have $B_1(X)=C(X)=X=X^{**}$. In particular, the topologies $\mu(X^*,X)$, $\mu(X^*,C(X))$ and $\mu(X^*,B_1(X))$ coincide with the norm topology.
It follows that $\Imn(A)=\Ibl(A)=\overline{A}^{\norm{\cdot}}=A$. 

On the other hand, for any uncountable $B\subset A$ we have $0\in\wscl{A} (=\overline{A}^w)$. Thus $0\in\Icl(A)$. We deduce that
$$\Imn(A)\subsetneqq\Icl(A)
\mbox{ and }\Icl(A)\not\subset\Ibl(A).$$

(3) Let $X=c_0(\Gamma)$ for an uncountable set $\Gamma$. Then $X^*=\ell_1(\Gamma)$ and $X^{**}=\ell_\infty(\Gamma)$. Moreover, $B_1(X)=C(X)=\ell_\infty^c(\Gamma)$, the subset of $\ell_\infty(\Gamma)$ formed by element with countable support. Since $C(X)\subsetneqq X^{**}$, we deduce that the topology $\mu(X^*,C(X))$ is strictly weaker than the norm topology.  Indeed, $\mu(X^*,C(X))$ is the topology of uniform convergence on absolutely convex $\sigma(C(X),X^*)$-compact subsets of $C(X)$. Since any $\sigma(C(X),X^*)$-compact subset of $C(X)$ is weak$^*$-compact (hence bounded) as a  subset of $X^{**}$, we deduce that $\mu(X^*,C(X))$ is  weaker than the norm topology. Since the dual of $(X^*,\mu(X^*,C(X)))$ is $C(X)\subsetneqq X^{**}$, we deduce that $\mu(X^*,C(X))$ must be strictly weaker. It follows that there is a normed-closed subset $A\subset X^*$ which is not $\mu(X^*,C(X))$-closed, hence
$$\overline{A}^{\norm{\cdot}}\subsetneqq \overline{A}^{\mu(X^*,C(X)}.$$

\end{example2}

Next we focus on the closure operators applied to convex sets. It is natural to ask whether the results of the above closure operators applied to a convex set are again convex. It seems not to be clear. Anyway, the situation in the convex case is easier, in particular, we have the following abstract result.

\begin{lemma}
Let $X$ be a vector space. Let $\alpha,\beta$ be two mappings with properties from Lemma~\ref{L:abstract Icl}. Let $\gamma_1$ and $\gamma_2$ be the closure operators defined as in Lemma~\ref{L:abstract Icl}. Then the following assertions are valid.
\begin{enumerate}[\rm(a)]
    \item Assume that for any $A\subset X$ we have
    $$\forall\lambda\in\ef\colon \alpha(\lambda A)=\lambda\alpha(A)\ \&\ \beta(\lambda A)=\lambda\beta(A).$$
    Then 
    for any $A\subset X$ we have
    $$\forall\lambda\in\ef\colon \gamma_1(\lambda A)=\lambda\gamma_1(A)\ \&\ \gamma_2(\lambda A)=\lambda\gamma_2(A).$$

\item Assume that $\alpha,\beta$ are translation invariant, i.e., $\alpha(x+A)=x+\alpha(A)$ and $\beta(x+A)=x+\beta(A)$ for each $x\in X$ and $A\subset X$. Then 
$$\gamma_j(A)+\gamma_j(B)\subset\gamma_j(\gamma_j(A+B)),\quad A,B\subset X, j=1,2.$$
\end{enumerate}
\end{lemma}

\begin{proof}
(a) Let us provide a proof for $\gamma_2$. The case of $\gamma_1$ is completely analogous.

Fix $\lambda\in\ef$. Assume $A=\bigcup_n A_n$. Then $\lambda A=\bigcup_n \lambda A_n$ and hence
$$\gamma_2(\lambda A)\subset \beta\left(\bigcup_n \alpha(\lambda A_n)\right)
=\lambda\beta\left(\bigcup_n \alpha(A_n)\right).$$
Since $(A_n)$ is an arbitrary cover of $A$, we deduce that $\gamma_2(\lambda A)\subset \lambda \gamma_2(A)$. The converse inclusion is obvious for $\lambda=0$. If $\lambda\ne 0$, then
$$\lambda \gamma_2(A)=\lambda \gamma_2\left(\frac1\lambda\cdot\lambda A\right)\subset \lambda\cdot\frac1\lambda \gamma_2(\lambda A)=\gamma_2(\lambda A).$$
This completes the proof of assertion (a).

(b) If $\alpha,\beta$ are translation invariant, clearly $\gamma_1,\gamma_2$ are translation invariant as well. So, for any $A,B\subset X$ we have
$$A+\gamma_j(B)=\bigcup_{x\in A} (x+\gamma_j(B))=\bigcup_{x\in A} \gamma_j(x+B)\subset\gamma_j\left(\bigcup_{x\in A} x+B\right)=\gamma_j(A+B),$$
hence
$$\gamma_j(A)+\gamma_j(B)\subset\gamma_j(\gamma_j(A)+B)\subset\gamma_j(\gamma_j(A+B)).$$
\end{proof}

\begin{prop}\label{P:convex comparison}
Let $X$ be a Banach space and let $A\subset X^*$ be a convex set.
\begin{enumerate}[\rm(a)]
\item The following inclusions hold:
    $$\begin{array}{ccccccccc}
        \Icl(A)&=&\Iccl(A)&\subset&\Iw(A)&\subset&\cl_{IF}(A)&&   \\
        \bigcup &&\bigcup&&&\sesup&&& \\
        \Imn(A) &&\cl_{IC}(A)&& \Ienv(A)&\subset&\overline{A}^{\sigma(X^*,B_1(X))} &\subset&\wscl{A} \\
        \bigcup &\sesub&&\sesub&&&\bigcup&& \\
        \overline{A}^{\norm{\cdot}}&\subset&\overline{A}^{\mu(X^*,C(X))}&\subset&\Ibl(A)&\subset&\Ibcl(A)&& 
    \end{array}
    $$
    \item If $A$ is norm-separable, then the following inclusions hold:
    $$\begin{array}{ccccccccc}
        \Icl(A)&=&\Iccl(A)&=&\Iw(A)&=&\cl_{IF}(A)&&   \\
        \| &&\|&&&\seeq&&& \\
        \Imn(A) &&\cl_{IC}(A)&=& \Ienv(A)&\subset&\overline{A}^{\sigma(X^*,B_1(X))} &\subset&\wscl{A} \\
        \| &\seeq&&\seeq&&&\bigcup&& \\
        \overline{A}^{\norm{\cdot}}&=&\overline{A}^{\mu(X^*,C(X))}&=&\Ibl(A)&=&\Ibcl(A)&& 
    \end{array}
    $$
    \item If $X$ does not contain a copy of $\ell_1$, then the following inclusions hold:
    $$\begin{array}{ccccccccc}
        \Icl(A)&=&\Iccl(A)&\subset&\Iw(A)&\subset&\cl_{IF}(A)&&   \\
        \bigcup &&\bigcup&&&&\|&& \\
        \Imn(A) &&\cl_{IC}(A)&& \Ienv(A)&=&\overline{A}^{\sigma(X^*,B_1(X))} &\subset&\wscl{A} \\
        \bigcup &\sesub&&\seeq&&&\|&& \\
        \overline{A}^{\norm{\cdot}}&\subset&\overline{A}^{\mu(X^*,C(X))}&=&\Ibl(A)&=&\Ibcl(A)&& 
    \end{array}
    $$
\end{enumerate}
\end{prop}

\begin{proof}
(a) This follows from Proposition~\ref{P:comparison of closure operators}(a) and Proposition~\ref{P:Ienv between closures}.

(b) In this case we have $\Ienv(A)=\overline{A}^{\norm{\cdot}}$ (see Proposition~\ref{P:comparison of closure operators}(b)). Hence, using (a) we get most of the equalities, except for the second and the third one on the last row.

Let us prove $\Ibcl(A)\subset\overline{A}^{\norm{\cdot}}$. Let $(x_n^*)$ be a norm dense sequence in $A$. Fix $\varepsilon>0$. Then
$A\subset \bigcup_n(x_n^*+\varepsilon B_{X^*})$. It follows that
$$\Ibcl(A)\subset \overline{\bigcup_n(x_n^*+\varepsilon B_{X^*})}^{\sigma(X^*,C(X))}\subset \overline{A+\varepsilon B_{X^*}}^{\sigma(X^*,C(X))}\subset \Ienv(A+\varepsilon B_{X^*}).$$
Further,
$$A+\varepsilon B_{X^*}\subset \bigcup_n (x_n^*+2\varepsilon B_{X^*}),$$
hence 
$$\Ienv(A+\varepsilon B_{X^*})\subset \overline{\co  \bigcup_n (x_n^*+2\varepsilon B_{X^*})}^{\norm{\cdot}}\subset\overline{A+2\varepsilon B_{X^*}}^{\norm{\cdot}}.$$
Since $\varepsilon>0$ is arbitrary, we deduce that $\Ibcl(A)\subset\overline{A}^{\norm{\cdot}}$ which completes the proof.

(c) We start by using (a).  Further, by Theorem~\ref{T:X not containing ell1} we have $B_1(X)=C(X)$. Hence $\overline{A}^{\mu(X^*,C(X))}=\overline{A}^{\sigma(X^*,B_1(X))}$, which proves the  equalities on the second and the third rows.

By Theorem~\ref{T:X not containing ell1} we further get
and $\Ienv(B)=\overline{\co B}^{\sigma(X^*,B_1(X))}$ for any $B\subset X^*$. Hence $\overline{B}^{\sigma(X^*,B_1(X))}\subset\Ienv(B)$ for each $B\subset X^*$. By Proposition~\ref{P:eight operators}(c) we deduce that $\overline{B}^{\sigma(X^*,B_1(X))}\subset\cl_{IF}(B)$ for $B\subset X^*$. 

Therefore, if $A$ is convex, then
$$\Ienv(A)=\overline{A}^{\sigma(X^*,B_1(X))}\subset\cl_{IF}(A)\subset\Ienv(A),$$
thus the remaining equality holds.
\end{proof}
 
In the next example we show that in case $A$ is convex and norm separable, no more equalities hold.

\begin{example2}\label{ex:convex separable}
(1) Let $X=\ell_1$ and $A=c_0\subset\ell_\infty=(\ell_1)^*$. Then 
$$\Ienv(A)=c_0\subsetneqq\ell_\infty=\overline{A}^{\sigma(X^*,B_1(X))}$$
by Example~\ref{ex:ell1}.

(2) Let $X=c_0$. Then $X^*=\ell_1$ and $X^{**}=\ell_\infty$. Let $\varphi\in X^{**}\setminus X$ and $A=\Ker\varphi$. Then
$$\overline{A}^{\sigma(X^*,B_1(X))}=A\subsetneqq \ell_1=\wscl{A}.$$
\end{example2}

We continue by collecting examples in duals to spaces not containing $\ell_1$.

\begin{example2}
(1) To show that we may have $\overline{A}^{\sigma(X^*,B_1(X))}\subsetneqq\wscl{A}$ we may use Example~\ref{ex:convex separable}(2).

(2) Let $X=c_0(\Gamma)$ for an uncountable set $\Gamma$, as in Example~\ref{exa:nonseparable}(3). Then $X^*=\ell_1(\Gamma)$, $X^{**}=\ell_\infty(\Gamma)$ and 
$$B_1(X)=C(X)=\ell_\infty^c(\Gamma)=\{g\in\ell^\infty(\Gamma)\setsep \spt g\mbox{ is countable}\}.$$ Let
$$A=\left\{f\in\ell_1(\Gamma)\setsep \sum_\gamma f(\gamma)=0\right\}.$$
Then $A$ is a norm closed linear subspace which is not $\mu(X^*,C(X))$-closed. Hence, $\overline{A}^{\mu(X^*,C(X))}=X^*$ (note that $A$ is a hyperplane and its closure is a linear subspace). Thus 
$$\overline{A}^{\norm{\cdot}}\subsetneqq \overline{A}^{\mu(X^*,C(X))}.$$

Next we observe that $\Icl(A)=X^*$. To this end take $f\in X^*\setminus A$. Then
$$c=\sum_\gamma f(\gamma) \ne0$$
and $\spt f$ is countable. Further,
$$B=\{f-c e_\gamma\setsep \gamma\in \Gamma\setminus \spt f\}\subset A.$$
If $A=\bigcup A_n$, then there is some $n\in\en$ with $B\cap A_n$ uncountable. Then
$$f\in\wscl{A_n\cap B}\subset\wscl{A_n}.$$
Since the cover $(A_n)$ is arbitrary, we deduce $f\in\Icl(A)$.

We continue by showing that $\Imn(A)=A$. To this end take $f\in X^*\setminus A$. Then
$$c=\sum_\gamma f(\gamma) \ne0.$$
Fix an arbitrary $\varepsilon\in(0,\frac{\abs{c}}4)$. Find a finite set $F\subset\Gamma$ with
$$\sum_{\gamma\in\Gamma\setminus F}\abs{f(\gamma)}\le\varepsilon.$$
For $n\in\en$ set
$$Z_n=\left\{g\in\ell_1(\Gamma)\setsep \exists H\subset\Gamma\colon \card H\le n\ \&\ \sum_{\gamma\in\Gamma\setminus H}\abs{g(\gamma)}<\varepsilon\right\}$$
Then $(Z_n)$ is an increasing cover of $X^*$ and, moreover, each $Z_n$ is weak$^*$ closed. (This easily follows from the fact that $g\in Z_n$ if and only if $g=g_1+g_2$ such that $\card\spt g_1\le n$ and $\norm{g_2}\le\varepsilon$.) We are going to show that
$$f\notin \overline{\bigcup_n \overline{A\cap Z_n}^{\mu(X^*,X)}}^{\norm{\cdot}}.$$

Note that the set $\{0\}\cup\{e_\gamma\setsep\gamma\in\Gamma\}$ is weakly compact in $X=c_0(\Gamma)$, hence the norm $\norm{\cdot}_\infty$ is $\mu(X^*,X)$-continuous on $X^*=\ell_1(\Gamma)$. Further, for each $k\in\en$ set
$$p_k(g)=\sup\left\{\sum_{\gamma\in H}\abs{g(\gamma)}\setsep H\subset\Gamma,\card\Gamma\le k\},\quad g\in\ell_1(\Gamma)\right\}.$$
Then $p_k$ is a norm on $\ell_1(\Gamma)$ which is equivalent to $\norm{\cdot}_\infty$ (note that $\norm{\cdot}_\infty\le p_k\le k\cdot\norm{\cdot}_\infty$), hence it is also $\mu(X^*,X)$-continuous.

Set $m=\card F$. 
Assume that $g\in A\cap Z_n$. Then there is $H\subset \Gamma$ with $\card H\le n$ such that $\sum_{\gamma\in\Gamma\setminus H}\abs{g(\gamma)}<\varepsilon$. Then
$$\begin{aligned}
\sum_{\gamma\in F\cup H}\abs{f(\gamma)-g(\gamma)}&\ge
\abs{\sum_{\gamma\in F\cup H}(f(\gamma)-g(\gamma))} \ge 
\abs{\sum_{\gamma\in \Gamma}(f(\gamma)-g(\gamma))}
-\abs{\sum_{\gamma\in \Gamma\setminus F\cup H}(f(\gamma)-g(\gamma))}\ge c-2\varepsilon,
\end{aligned}$$
hence
$p_{m+n}(f-g)\ge c-2\varepsilon$, Since $g\in A\cap Z_n$ is arbitrary, we get
$$p_{m+n}-\dist (f,A\cap Z_n)\ge c-2\varepsilon.$$
Since $p_{m+n}$ is $\mu(X^*,X)$-continuous, we deduce that
$$p_{m+n}-\dist (f,\overline{A\cap Z_n}^{\mu(X^*,X)})\ge c-2\varepsilon.$$
SInce $p_{m+n}\le \norm{\cdot}_1$, we deduce that
$$\norm{\cdot}_1-\dist (f,\overline{A\cap Z_n}^{\mu(X^*,X)})\ge c-2\varepsilon,$$
hence
$$\norm{\cdot}_1-\dist \left(f,\bigcup_n\overline{A\cap Z_n}^{\mu(X^*,X)}\right)\ge c-2\varepsilon,$$
so $f\notin \Imn(A)$.

We thus have 
$$\Imn(A)\subsetneqq\Icl(A).$$

Further, similarly as $\Icl(A)=X^*$ we may prove $\cl_{IC}(A)=X^*$.

Summarizing, we have
$$\overline{A}^{\norm{\cdot}}\subsetneqq \overline{A}^{\mu(X^*,C(X))}, \overline{A}^{\norm{\cdot}}\subsetneqq \cl_{IC}(A), \Imn(A)\subsetneqq\Icl(A)$$
and also
$$\cl_{IC}(A)\not\subset\Imn(A), \overline{A}^{\mu(X^*,C(X))}\not\subset\Imn(A).$$

(3) Let $X=\C([0,\omega_1])$, the space of continuous functions on the ordinal interval $[0,\omega_1]$. Then
$X^*=\ell_1([0,\omega_1])$ (note that $[0,\omega_1]$ is a scattered compact space and hence each Radon measure on it is countably supported) and $X^{**}=\ell_\infty([0,\omega_1])$. Observe that
$$B_1(X)=C(X)=\{f\in \ell_\infty([0,\omega_1])\setsep f \mbox{ is continuous at }\omega_1\}.$$
Set 
$$A=\{g\in\ell_1([0,\omega_1])\setsep g(\omega_1)=0\}.$$
Then $A$ is a norm-closed hyperplane in $X^*$. It is not $\mu(X^*,C(X))$-closed (since the characteristic function of $\{\omega_1\}$ does not belong to $C(X)$), hence $\overline{A}^{\mu(X^*,C(X))}=X^*$. 

Next we are goint to show that $\Imn(A)=X^*$. To this end fix $g\in X^*\setminus A$. Then $g=g_0+c\delta_{\omega_1}$, where $g_0$ is supported by $[0,\omega_1)$ and $c\ne 0$. Fix $\gamma<\omega_1$ such that $g_0|_{(\gamma,\omega_1]}=0$. Set
$$\Gamma=\{\alpha\in(\gamma,\omega_1)\setsep \alpha\mbox{ is a limit ordinal}\}$$
Then $\Gamma$ is a closed unbounded set in $[0,\omega_1)$ and
$$B=\{g_0+c\delta_{\alpha}\setsep \alpha\in\Gamma\}\subset A.$$
Let $(A_n)$ be a countable cover of $A$. It follows from \cite[Theorem 8.3]{jech} there is some $n\in\en$ such that the set
$$\Gamma_n=\{\alpha\in\Gamma\setsep g_0+c\delta_\alpha \in A_n\}$$
is stationary (i.e., it intersects each closed unbounded set, see, e.g., \cite[Definition 8.1]{jech}). We claim that 
$$g\in\overline{B\cap A_n}^{\mu(X^*,X)}.$$
Assume not. Then there is a weakly compact set $L\subset \C([0,\omega_1])$ such that
$$\inf_{h\in B\cap A_n}\sup_{f\in L}\abs{\ip{h-g}{f}} >0.$$
Fix $\varepsilon>0$ such that
$$\inf_{h\in B\cap A_n}\sup_{f\in L}\abs{\ip{h-g}{f}} >\varepsilon.$$
Then for each $\alpha\in\Gamma_n$ there is some $f_\alpha\in L$ such that
$$\abs{\ip {g_0+c\delta_\alpha-g}{f_\alpha}}>\varepsilon,$$
i.e.,
$$\abs{c(f_\alpha(\alpha)-f_\alpha(\omega_1))}>\varepsilon.$$
Since each $\alpha\in\Gamma_n$ is a limit ordinal and each $f_\alpha$ is a continuous function, there is $\gamma_\alpha<\alpha$ such that
$$\abs{c(f_\alpha(\gamma)-f_\alpha(\omega_1))}>\varepsilon \mbox{ for }\gamma\in(\gamma_\alpha,\alpha].$$
By Fodor's Pressing down lemma (see, e.g., \cite[Theorem 8.7]{jech}) there is a stationary set $\Gamma'\subset\Gamma_n$ and $\beta<\omega_1$ such that
$$\gamma_\alpha=\beta\mbox{ for }\alpha\in\Gamma'.$$
We now construct a sequence $(\alpha_n)$ in $\Gamma'$ as follows:
\begin{itemize}
    \item $\alpha_1\in\Gamma'$ is arbitrary.
    \item Given $\alpha_n$, we find $\alpha_{n+1}\in \Gamma'$ such that $\alpha_{n+1}>\alpha_n$ and $f_{\alpha_n}|_{[\alpha_{n+1},\omega_1]}$ is constant.
\end{itemize}
Since $L$ is weakly compact, the sequence $(f_{\alpha_n})$ has a subsequence pointwise converging to a continuous function $f$.
Let $\alpha=\sup_n\alpha_n$, Then $f_{[\alpha,\omega_1]}$ is constant and $\abs{f(\theta)-f(\omega_1)}\ge \varepsilon$ for $\theta\in(\beta,\alpha)$. Since $\alpha$ is a limit ordinal, this contradicts the continuity if $f$.

So, we have proved that $g\in\overline{B\cap A_n}^{\mu(X^*,X)}$. Since the cover $(A_n)$ was arbitrary, we deduce that $g\in\Imn(A)$.
Thus
$$\overline{A}^{\norm{\cdot}}\subsetneqq \Imn(A).$$

We note that $\cl_{IC}(A)=X^*$ as well. Indeed, assume that $g\in X^*\setminus A$. Let $\Gamma$ and $B$ be as above. If $A=\bigcup_n A_n$, there is some $n\in \en$ such that $B\cap A_n$ is uncountable. Further, let $A_n=\bigcup_m A_{n,m}$. Then there is some $m$ such that $B\cap A_{m,n}$ is uncountable. Then $g\in\wscl{B\cap A_{m,n}}$. Since $(A_{n,m})$ is an arbitrary cover of $A_n$, we deduce that $g\in\Icl(A_n)\subset\Ienv(A_n)$. Since $(A_n)$ is an arbitrary cover of $A$, we conclude that $g\in\cl_{IC}(A)$.
\end{example2} 

The next example shows, that the seemingly natural closure operator $\Icl$ (which equals $\Iccl$) in general does not yield the positive answer to Question~\ref{q1}. However, we point out that we do not know whether in this case $\Icl$ coincides with $\cl_{IF}$.

\begin{example2} \label{ex:Pd[0,1]}
Let $X=\C([0,1])$ and let $A$ be the set of countably supported probability measures on $[0,1]$, considered as a subset of $\C([0,1])^*$.

Since $A$ contains all Dirac measures which are exactly extreme points of $\P([0,1])$, by \cite[Theorem 2.3]{FL} we deduce that $\Ienv(A)=\wscl{A}=\P([0,1])$.

We are going to show that $\Icl(A)=A$. To this end, given $n\in\en$, denote by $A_n$ the set of probabilities supported by a set of cardinality at most $n$. Then each $A_n$ is weak$^*$-compact and $\bigcup_n A_n$ is norm dense in $A$. Hence, for any $\varepsilon>0$ we have
$$\Icl(A)\subset \overline{\bigcup_n A_n+\varepsilon B_{X^*}}^{\norm{\cdot}}
\subset  \overline{A+\varepsilon B_{X^*}}^{\norm{\cdot}}\subset A+2\varepsilon B_{X^*},$$
hence $\Icl(A)\subset\overline{A}^{\norm{\cdot}}=A$.
\end{example2}

\begin{ques}\label{q:remaining}
Let $X$ be a Banach space and $A\subset X^*$. Do the equalities
$$\cl_{IC}(A)=\Iccl(A),\quad \Iw(A)=\cl_{IF}(A), \quad \overline{A}^{\mu(X^*,C(X))}=\Ibl(A)$$
hold?
\end{ques}

\section{Final remarks and open problems}

The answer to the `locally convex' variant of Question~\ref{q1} is for some Banach spaces positive and for some Banach spaces negative as witnessed by the results of Section~\ref{sec:LCS}. However, a characterization of those Banach spaces for which the answer is positive  is still missing as pointed in Question~\ref{q:ell1}.
 
The `topological' variant of Question~\ref{q1} is widely open. We have no counterexample to it, but we do not have either any example for which the `topological' variant has a positive answer but the `locally convex' variant has a negative answer. In view of Theorem~\ref{T:characterization} the following problem is natural to ask:

\begin{ques}
Assume that $X$ is a Banach space and $A,B\subset X^*$ are two convex sets such that $A\cup B$ is also convex. Is
$\Ienv(A\cup B)=\Ienv(A)\cup\Ienv(B)$?
\end{ques}

Further, Question~\ref{q:bounded} suggests that there may be a substantial difference between (I)-envelopes of bounded and unbounded sets. It seems that the following question is interesting and open.

\begin{ques}
Let $X$ be a Banach space and $A\subset X^*$. Is the set
$$\Ienv_b (A)=\bigcup_{n\in\en} \Ienv (A\cap n B_{X^*})$$
norm-dense in $\Ienv(A)$?
\end{ques}

Note that obviously $\Ienv_b(A)=\Ienv (A)$ if $A$ is bounded. For unbounded set $A$ this equality may or may not hold as witnessed by the following example.

\begin{example2}
(1) Let $X=(\ell^\infty)^*$ and $A=\ell^\infty$ canonically embedded into $X^*=(\ell^\infty)^{**}$. It follows from \cite[Example 4.1]{Ienv-israel} that $\Ienv_b(A)=\Ienv(A)=X^*$.

(2) More generally, let $Y$ be a $c$-Grothendieck space for some $c\ge1$ (see \cite{HB-groth}). Let $X=Y^*$ and consider $Y$ canonically embedded into $X^*=Y^{**}$. It follows from 
\cite[Proposition 2.2]{HB-groth} that  $\Ienv_b(Y)=\Ienv(Y)=X^*$.

Examples of $1$-Grothendieck spaces include $\ell^\infty(\Gamma)$ for any set $\Gamma$, $L^\infty(\mu)$ for a $\sigma$-finite measure $\mu$ and quotients of these spaces (see \cite[Section 4]{HB-groth}) and also $\C(K)$ spaces from the class addressed in \cite{lechner-1gr}.

(3) Let $Y$ be a Grothendieck space which is not $c$-Grothendieck for any $c\ge1$. Such a space is provided in \cite[Theorem 1.2]{HB-groth}. Let $X=Y^*$ and consider $Y$ canonically embedded into $X^*=Y^{**}$. By \cite[Theorem 2.1]{Ienv-studia} we have
$\Ienv(Y)=X^*$. However, it follows from \cite[Proposition 2.2]{HB-groth} that  $\Ienv_b(Y)\subsetneqq X^*$ as $\Ienv_b(Y)$ is of first category in $X^*$. It seems not to be clear whether $\Ienv_b(Y)$ is norm-dense in $X^*$.
\end{example2}

Another problem concerns a quantitative version of Lemma~\ref{L:separ}. If $X$ is a Banach space and $A\subset X^*$, define for $\eta\in X^*$ the following two quantities:
$$\begin{aligned}
d_1(\eta)&=\dist(\eta,\Ienv(A)),\\
d_2(\eta)&=\sup\left\{\inf_{n\in\en}\Re\eta(x_n)-\sup_{\xi\in A}\limsup_{n\to\infty}\Re\xi(x_n)\setsep (x_n)\subset B_X\right\}.
\end{aligned}$$
It follows from Lemma~\ref{L:separ} that $d_1(\eta)=0$ if and only if $d_2(\eta)=0$. It is easy to check that $d_2(\eta)\le d_1(\eta)$: If $\zeta\in\Ienv(A)$ and $(x_n)\subset B_X$, then
$$\begin{aligned}
\inf_{n\in\en}\Re\eta(x_n)-&\sup_{\xi\in A}\limsup_{n\to\infty}\Re\xi(x_n)\\&
=\inf_{n\in\en}\Re\eta(x_n)-\limsup_{n\to\infty}\Re\zeta(x_n)+\limsup_{n\to\infty}\Re\zeta(x_n)-\sup_{\xi\in A}\limsup_{n\to\infty}\Re\xi(x_n)\\
&\le \inf_{n\in\en}\Re\eta(x_n)-\limsup_{n\to\infty}\Re\zeta(x_n)
\le \liminf_{n\to\infty} \Re (\eta(x_n)-\zeta(x_n))\le \norm{\eta-\zeta}.
\end{aligned}$$
It follows that $d_2(\eta)\le \norm{\eta-\zeta}$ for any $\zeta\in\Ienv(A)$, so indeed $d_2(\eta)\le d_1(\eta)$. 
Validity of a kind of converse is an open problem.

\begin{ques} 
Is there a universal constant $c\ge1$ such that $d_1\le c\cdot d_2$?
\end{ques}

It seems that a solution to this question would substantially
deepen the understanding of (I)-envelopes. It follows from the proof of Lemma~\ref{L:separ} that
$$d_2(\eta)=\inf \left\{ \dist \left(\eta,\overline{\co\bigcup_{n=1}^\infty \wscl{\co A_n}}^{\norm{\cdot}}\right)\setsep A=\bigcup_{n=1}^\infty A_n
\right\}.$$
It easily follows that $d_1=d_2$ whenever $A$ is norm-separable, but the general case is widely open.

We finish by formulating the following intriguing problem on (I)-envelopes:

\begin{ques}
Let $X$ be a separable Banach spaces. Is the dual unit ball $B_{X^*}$ equal to the (I)-envelope of the set of its weak$^*$-exposed points?
\end{ques}

This question is closely related to the problem asked by G.~Godefroy in \cite[Question V.1]{godefroy87}. Indeed, the quoted question asks whether the norm-closed convex hull of the set of weak$^*$-exposed points of $B_{X^*}$ is the whole unit ball provided $X$ is separable and does not contain an isomorphic copy of $\ell^1$. In view of Theorem~\ref{T:X not containing ell1} (see also \cite[Proposition 2.2(b)]{FL}) it is a special case of our last question.

\bibliography{ienv}\bibliographystyle{acm}

\def\cprime{$'$}
\begin{thebibliography}{10}

\bibitem{Albiac-Kalton}
{\sc Albiac, F., and Kalton, N.~J.}
\newblock {\em Topics in {B}anach space theory}, second~ed., vol.~233 of {\em
  Graduate Texts in Mathematics}.
\newblock Springer, [Cham], 2016.
\newblock With a foreword by Gilles Godefory.

\bibitem{HB-groth}
{\sc Bendov\'{a}, H.}
\newblock Quantitative {G}rothendieck property.
\newblock {\em J. Math. Anal. Appl. 412}, 2 (2014), 1097--1104.

\bibitem{engelking}
{\sc Engelking, R.}
\newblock {\em General topology}, second~ed., vol.~6 of {\em Sigma Series in
  Pure Mathematics}.
\newblock Heldermann Verlag, Berlin, 1989.
\newblock Translated from the Polish by the author.

\bibitem{Fab-BST}
{\sc Fabian, M., Habala, P., H\'{a}jek, P., Montesinos, V., and Zizler, V.}
\newblock {\em Banach space theory}.
\newblock CMS Books in Mathematics/Ouvrages de Math\'{e}matiques de la SMC.
  Springer, New York, 2011.
\newblock The basis for linear and nonlinear analysis.

\bibitem{FL}
{\sc Fonf, V.~P., and Lindenstrauss, J.}
\newblock Boundaries and generation of convex sets.
\newblock {\em Israel J. Math. 136\/} (2003), 157--172.

\bibitem{godefroy87}
{\sc Godefroy, G.}
\newblock Boundaries of a convex set and interpolation sets.
\newblock {\em Math. Ann. 277}, 2 (1987), 173--184.

\bibitem{james}
{\sc James, R.~C.}
\newblock Bases and reflexivity of {B}anach spaces.
\newblock {\em Ann. of Math. (2) 52\/} (1950), 518--527.

\bibitem{jech}
{\sc Jech, T.}
\newblock {\em Set theory}.
\newblock Springer Monographs in Mathematics. Springer-Verlag, Berlin, 2003.
\newblock The third millennium edition, revised and expanded.

\bibitem{Ienv-israel}
{\sc Kalenda, O. F.~K.}
\newblock ({I})-envelopes of closed convex sets in {B}anach spaces.
\newblock {\em Israel J. Math. 162\/} (2007), 157--181.

\bibitem{Ienv-studia}
{\sc Kalenda, O. F.~K.}
\newblock ({I})-envelopes of unit balls and {J}ames' characterization of
  reflexivity.
\newblock {\em Studia Math. 182}, 1 (2007), 29--40.

\bibitem{kothe}
{\sc K\"{o}the, G.}
\newblock {\em Topological vector spaces. {I}}.
\newblock Translated from the German by D. J. H. Garling. Die Grundlehren der
  mathematischen Wissenschaften, Band 159. Springer-Verlag New York Inc., New
  York, 1969.

\bibitem{lechner-1gr}
{\sc Lechner, J.}
\newblock 1-{G}rothendieck {$C(K)$} spaces.
\newblock {\em J. Math. Anal. Appl. 446}, 2 (2017), 1362--1371.

\bibitem{LMNS}
{\sc Luke\v{s}, J., Mal\'{y}, J., Netuka, I., and Spurn\'{y}, J.}
\newblock {\em Integral representation theory}, vol.~35 of {\em De Gruyter
  Studies in Mathematics}.
\newblock Walter de Gruyter \& Co., Berlin, 2010.
\newblock Applications to convexity, Banach spaces and potential theory.

\bibitem{LMZ}
{\sc Luke\v{s}, J., Mal\'{y}, J., and Zaj\'{\i}\v{c}ek, L.}
\newblock {\em Fine topology methods in real analysis and potential theory},
  vol.~1189 of {\em Lecture Notes in Mathematics}.
\newblock Springer-Verlag, Berlin, 1986.

\bibitem{miljutin}
{\sc Miljutin, A.~A.}
\newblock Isomorphism of the spaces of continuous functions over compact sets
  of the cardinality of the continuum.
\newblock {\em Teor. Funkci\u{\i} Funkcional. Anal. i Prilo\v{z}en. Vyp. 2\/}
  (1966), 150--156. (1 foldout).

\bibitem{odell-rosenthal}
{\sc Odell, E., and Rosenthal, H.~P.}
\newblock A double-dual characterization of separable {B}anach spaces
  containing {$l^{1}$}.
\newblock {\em Israel J. Math. 20}, 3-4 (1975), 375--384.

\bibitem{simons-minimax}
{\sc Simons, S.}
\newblock Maximinimax, minimax, and antiminimax theorems and a result of {R}.
  {C}. {J}ames.
\newblock {\em Pacific J. Math. 40\/} (1972), 709--718.

\end{thebibliography}
\end{document}